\documentclass[reqno, 11pt, letterpaper]{amsart}

\usepackage{amsmath}
\usepackage{amssymb}
\usepackage{graphicx}
 \usepackage{subfig}
\usepackage{amsthm,graphicx,color,yfonts}
\usepackage{hyperref}

\newtheorem{theorem}{Theorem}

\newtheorem{proposition}[theorem]{Proposition}
\newtheorem{lemma}[theorem]{Lemma}
\newtheorem{corollary}[theorem]{Corollary}

\theoremstyle{remark}
\newtheorem{remark}[theorem]{Remark}

\newcommand{\R}{\mathbb{R}}

\newcommand{\Z}{\mathbb{Z}}

\usepackage{wrapfig}
\usepackage{tikz}
\usetikzlibrary{arrows,calc,decorations.pathreplacing}
\definecolor{light-gray1}{gray}{0.90}
\definecolor{light-gray2}{gray}{0.80}
\definecolor{light-gray3}{gray}{0.60}

\numberwithin{equation}{section}

\numberwithin{theorem}{section}

\numberwithin{table}{section}

\numberwithin{figure}{section}

\ifx\pdfoutput\undefined
  \DeclareGraphicsExtensions{.pstex, .eps}
\else
  \ifx\pdfoutput\relax
    \DeclareGraphicsExtensions{.pstex, .eps}
  \else
    \ifnum\pdfoutput>0
      \DeclareGraphicsExtensions{.pdf}
    \else
      \DeclareGraphicsExtensions{.pstex, .eps}
    \fi
  \fi
\fi

\title[Standing waves for NLS with a partial confinement]{Uniqueness and orbital stability of standing waves for the nonlinear Schr\"odinger equation with a partial confinement}

\date{\today}
\author[Y. Hong]{Younghun Hong}
\address{Department of mathematics, Chung-Ang University, Seoul 06974, Korea}
\email{yhhong@cau.ac.kr}

\author[S. Jin]{Sangdon Jin}
\address{Department of mathematics, Chung-Ang University, Seoul 06974, Korea}
\email{sdjin@cau.ac.kr}

\begin{document}

\maketitle

\begin{abstract}
We consider the 3d cubic nonlinear Schr\"odinger equation (NLS) with a strong 2d harmonic potential. The model is physically relevant to observe the lower-dimensional dynamics of the Bose-Einstein condensate, but its ground state cannot be constructed by the standard method due to its supercritical nature. In Bellazzini-Boussa\"id-Jeanjean-Visciglia \cite{BBJV}, a proper ground state is constructed introducing a constrained energy minimization problem. In this paper, we further investigate the properties of the ground state. First, we show that as the partial confinement is increased, the 1d ground state is derived from the 3d energy minimizer with a precise rate of convergence. Then, by employing this dimension reduction limit, we prove the uniqueness of the 3d minimizer provided that the confinement is sufficiently strong. Consequently, we obtain the orbital stability of the minimizer, which improves that of the set of minimizers in the previous work \cite{BBJV}.
\end{abstract}


\section{Introduction}
\subsection{Background}
Consider the 3d cubic nonlinear Schr\"odinger equation (NLS) with a 2d partial confinement 
\begin{equation}\label{NLS0}
i\partial_tu=(-\Delta_x+\omega^2|y|^2)u-\frac{1}{\omega}|u|^2u,
\end{equation}
where $u=u(t,x): I(\subset\mathbb{R})\times\mathbb{R}^3\to\mathbb{C}$ and
$$x=(y,z)\in\mathbb{R}_x^3=\mathbb{R}_y^2\times\mathbb{R}_z.$$
The parameter $\omega>0$ represents the strength of the 2d quadratic potential and the weakness of the nonlinearity simultaneously. The NLS is a canonical equation for wave propagation that arises in various fields of physics \cite{SulemSulem}. This particular model \eqref{NLS0} describes the mean-field dynamics of an extremely cooled boson gas, namely a Bose--Einstein condensate confined in an anisotropic trap. We refer to Chen \cite{XChen} for a rigorous derivation of the model \eqref{NLS0} from the many-body bosonic system.

In physical experiments, the 2d partial confinement $\omega^2|y|^2$ is used to simulate lower-dimensional cigar-shaped condensates \cite{KSFBCCCS}. By increasing the strength of the trap $\omega\to\infty$, a low energy state $u(t,x)$ to the 3d NLS \eqref{NLS0} can be asymptotically described by a factorized state 
\begin{equation}\label{3d solution approximation}
v(t,z)e^{-2it\omega}\sqrt{\omega}\Phi_0(\sqrt{\omega}y),
\end{equation}
where $v=v(t,z):I(\subset\mathbb{R})\times\mathbb{R}\to\mathbb{C}$ is a solution to the 1d cubic NLS
\begin{equation}\label{1dNLS}
i\partial_t v=-\partial_z^2 v-\frac{1}{2\pi}|v|^2v
\end{equation}
and $\Phi_0(y)=\frac{1}{\sqrt{\pi}}e^{-\frac{|y|^2}{2}}$ is an $L_y^2(\mathbb{R}^2)$-normalized eigenfunction of the 2d Hermite operator
$$H_y=-\Delta_y+|y|^2$$
corresponding to the lowest eigenvalue $2$ (refer to Appendix A for the proof).
\begin{figure}[h]
\subfloat{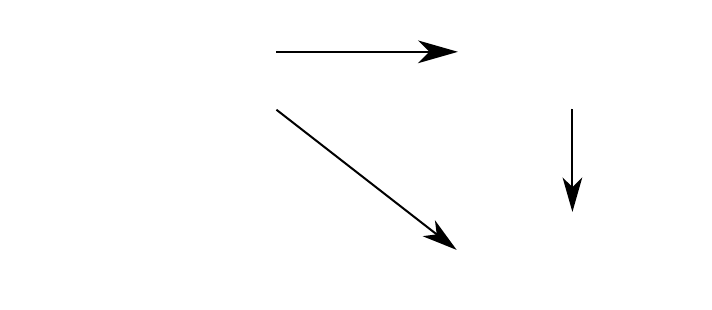}
\caption*{Figure 1.}
\label{Diagram}
\end{figure}

The dimension reduction of Bose--Einstein condensates is important both theoretically and experimentally, and it has been studied in various settings. In \cite{LSY, SY}, it was shown that cigar-shaped and disk-shaped condensates were obtained from the ground state of the 3d many-body bosonic Schr\"odinger operator. The lower-dimensional time-dependent NLSs are derived from the 3d linear Schr\"odinger equation (LS) including the attractive interaction case \cite{ Bossmann 1d, Bossmann 2d, Bossmann-Teufel, CH, CH 2d}. Moreover, the convergence from a high- to a low-dimensional NLS is established in various physical contexts \cite{AC,BM, BM1,  HT, MS}. We also note that a similar dimension reduction problem can be formulated on the product space $\R^k\times M$ as a compact manifold $M$ shrinks \cite{HP, HPTV, TTV}. For further references and related results, we refer to the survey article by Bao and Cai \cite{BaoCai}.

In analysis perspective, however, because the 3d cubic NLS \eqref{NLS} is mass-supercritical, one may encounter several technical challenges. It is not globally well-posed, and a finite time blow-up may occur (see Lemma \ref{blow up lemma}). Furthermore, an orbitally stable state cannot be constructed as a typical notion of a ground state,  because the energy functional is not bounded below under a mass constraint. Nevertheless, it is shown in the important work of Bellazzini, Boussa\"id, Jeanjean and Visciglia \cite{BBJV} that a certain energy minimizer can still be constructed by imposing an additional constraint. In addition, the authors established the orbital stability of a set of energy minimizers \cite[Theorem 1]{BBJV}.

The purpose of this study is twofold. First, we rigorously derive the 1d ground state from the dimension reduction limit ($\omega\to\infty$) of the 3d energy minimizer obtained in \cite{BBJV}. This corresponds to the downward arrow on the right-hand side of Figure 1 along the minimum energy states. Second, by employing convergence, we establish the local uniqueness of the 3d energy minimizer. Consequently, we upgrade the orbital stability of the set of energy minimizers to that of the minimizer.

\subsection{Setup and the statement of the main result}

To clarify the connection to the 1d model, we reformulate the setup of the problem as follows. Motivated by the ansatz \eqref{3d solution approximation}, replacing $e^{2it\omega}\frac{1}{\sqrt{\omega}}u(t,\frac{y}{\sqrt{\omega}},z)$ with $u(t,x)$, we rewrite \eqref{NLS0} as 
\begin{equation}\label{NLS}
i\partial_tu=\big(\omega (H_y-2)-\partial_z^2\big)u-|u|^2u.
\end{equation}
Then, it can be shown that the NLS \eqref{NLS} is locally well-posed in the weighted energy space
\begin{equation}\label{energy space}
\Sigma=\Big\{u\in H_x^1(\mathbb{R}^3) : yu\in L_x^2(\mathbb{R}^3)\Big\}
\end{equation}
equipped with the norm
$$\|u\|_{ {\Sigma} }: =\left\{\int_{\mathbb{R}^3}|u(x)|^2+|\nabla_x u(x)|^2+|y|^2|u(x)|^2 dx\right\}^{1/2}$$
(see \cite{Carles, Caz}), and its solutions preserve the mass
$$M(u)=\|u\|_{L_x^2(\mathbb{R}^3)}^2$$
and energy
$$E_\omega(u)=\frac{\omega}{2}\|u\|_{\dot{\Sigma}_y}^2+\frac{1}{2}\|\partial_z u\|_{L_x^2(\mathbb{R}^3)}^2-\frac{1}{4}\|u\|_{L_x^4(\mathbb{R}^3)}^4,$$
where
\begin{equation}\label{energy norm}
\|u\|_{\dot{\Sigma}_y}:=\|\sqrt{H_y-2}u\|_{L_x^2(\mathbb{R}^3)}=\left\{\int_{\mathbb{R}^3}|\nabla_y u(x)|^2+|y|^2|u(x)|^2-2|u(x)|^2dx\right\}^{1/2}.
\end{equation}

\begin{remark}\label{rmk: modified energy}
To be precise, the energy $E_\omega(u)$ is ``adjusted" in that a mass term $\omega M(u)$ is subtracted from the more ``natural" energy functional
$$\tilde{E}_\omega(u)=\frac{\omega}{2}\left\{\|\nabla_y u\|_{L_x^2(\mathbb{R}^3)}^2+\|yu\|_{L_x^2(\mathbb{R}^3)}^2\right\}+\frac{1}{2}\|\partial_z u\|_{L_x^2(\mathbb{R}^3)}^2-\frac{1}{4}\|u\|_{L_x^4(\mathbb{R}^3)}^4.$$ 
Note that this modification does not make any essential difference in the energy minimization problem because the mass is fixed (see \eqref{minimization} below). The role of $-\omega M(u)$ is to keep the ground state energy finite in the limit $\omega\to\infty$. Indeed, it is expected that the lowest energy state would be of the form $u(t,x)=v(t,z)\Phi_0(y)$. Then, the energy $\tilde{E}(u)$ grows as $\frac{\omega}{2}\|\sqrt{H_y}u\|_{L_x^2(\mathbb{R}^3)}^2=\omega M(u)$. Thus, $\omega M(u)$ is removed from the energy.
\end{remark}

For large $\omega\geq 1,$ which will be specified later, we consider the energy minimization problem with an additional constraint, 
\begin{equation}\label{minimization}
\mathcal{J}_\omega(m)=\inf\left\{E_\omega(u): u\in \Sigma,\   M(u)=m\textup{ and }\|u\|_{\dot{\Sigma}_y}^2\leq  \sqrt{\omega}\right\}.
\end{equation}
We now state the existence of a minimizer for the problem $\mathcal{J}_\omega(m)$.
 
\begin{theorem}[Existence of an energy minimizer, {\cite[Theorem 1]{BBJV}}]\label{th1} If $\omega \ge (C_{GN})^4m^2$, where $C_{GN}>0$ is the constant given in Lemma \ref{GN type inequality}, then the following hold:
\begin{enumerate}
\item For any minimizing sequence $\{u_n\}_{n=1}^\infty$ of the variational problem $\mathcal{J}_\omega(m)$, there exists a subsequence of $\{u_n\}_{n=1}^\infty$ (but still denoted by $\{u_n\}_{n=1}^\infty$), $\{\theta_n\}_{n=1}^\infty\subset\mathbb{R}$ and $\{z_n\}_{n=1}^\infty\subset\mathbb{R}$ such that $e^{i\theta_n}u_n(y,z-z_n)\to Q_\omega$ in $\Sigma$.
\item The limit $Q_\omega$ is a minimizer for the problem $\mathcal{J}_\omega(m)$.
\item A minimizer must be of the form $e^{i\theta}Q_\omega(y,z-z_0)$, with $\theta, z_0\in \R$, where $Q_\omega(x)=Q_\omega(|y|,|z|)$ and it is non-negative and decreasing as $|y|,|z|\to\infty$.
\item The minimizer $Q_\omega$ solves the elliptic equation
\begin{equation}\label{3D EL equation}
\omega(H_y  -2  )Q_\omega-\partial_z^2Q_\omega- Q_\omega^3=-\mu_\omega Q_\omega,
\end{equation} 
where $\mu_\omega\in \R$ is a Lagrange multiplier.
\end{enumerate}
\end{theorem}
 
\begin{remark}\label{remark: existence remark}
\begin{enumerate} 
\item   Theorem \ref{th1} is proved in \cite[Theorem 1]{BBJV} in a slightly different setting\footnote{By scaling $u=\sqrt{\omega}  v(y,\sqrt{\omega}z)$, the  problem \eqref{minimization} is equivalent to the small mass constraint energy minimization
$$\begin{aligned}
 \inf\left\{\mathcal{E}(v): v\in \Sigma,\   M(v)=\frac{m}{\sqrt{\omega}} \textup{ and }\|\nabla_y v\|_{L_x^2(\mathbb{R}^3)}^2+  \|y v\|_{L_x^2(\mathbb{R}^3)}^2\leq 1+{2m}  \right\},
\end{aligned}$$
where $\mathcal{E}(v)=\frac12\|\nabla_x v\|_{L_x^2(\mathbb{R}^3)}^2+ \frac12 \|y v\|_{L_x^2(\mathbb{R}^3)}^2- \frac14\|v\|_{L_x^4(\mathbb{R}^3)}^4$ is the energy functional without the parameter $\omega$. The variational problem considered in \cite{BBJV} is similar to this rescaled problem but under a slightly stronger additional constraint $\|\nabla_x v\|_{L_x^2(\mathbb{R}^3)}^2+  \|y v\|_{L_x^2(\mathbb{R}^3)}^2\leq 1+{2m}$, where the full gradient norm bound, not for the partial gradient $\nabla_y v$, is imposed. As it will be sketched in Section \ref{minexis}, this change requires only minor modifications in the proof.}. 
\item As observed in \cite{BBJV}, a minimizer for $\mathcal{J}_\omega(m)$ solves the usual Euler--Lagrange equation \eqref{3D EL equation} because it is not located on the boundary of the constraint, i.e., the sphere $\|u\|_{\dot{\Sigma}_y}^2= \sqrt{\omega}$, which is forbidden in the function space (see Corollary \ref{cor: forbidden region}).
\end{enumerate}
\end{remark} 
\begin{remark}
The additional constraint $\|u\|_{\dot{\Sigma}_y}^2\leq  \sqrt{\omega}$ seems natural for the supercritical problem \eqref{minimization}, because according to the strong blow-up conjecture (see \cite{HR}), it is expected that any negative energy solution to the Cauchy problem \eqref{NLS} with initial data $\|u_0\|_{\dot{\Sigma}_y}^2>\sqrt{\omega}$ would blow up in finite time. This means that the steady state $Q_\omega(x)e^{i\mu_\omega t}$ may occupy the least energy among all global-in-time solutions having the same mass. Indeed, it is easy to show finite-time blow up in the finite variance case, i.e., $xu_0\in L_x^2(\mathbb{R}^3)$ (see Lemma  \ref{blow up lemma}), but it is difficult to eliminate the finite variance assumption \cite{HR}.
\end{remark} 

As shown in \cite[Theorem 2]{BBJV}, the minimizer $Q_\omega$ is asymptotically reduced to the one-dimensional state in the sense that as $\omega\to\infty$,
\begin{equation}\label{BBJV reduction}
Q_\omega(x)-Q_{\omega,\parallel}(z)\Phi_0(y)\to 0,
\end{equation}
where $Q_{\omega,\parallel}(z)=\langle Q_\omega(\cdot,z),\Phi_0\rangle_{L_y^2(\mathbb{R}^2)}$ is the $\Phi_0(y)$-directional component. This justifies the dimension reduction of the Bose--Einstein condensates in physical experiments.

In this paper, with reference to the dimension reduction limit, we examine the connection to the 1d minimization problem
\begin{equation}\label{1d minimization}
\mathcal{J}_\infty(m)=\inf_{w\in H^1_z(\mathbb{R})}\left\{E_\infty(w) : \|w\|_{L_z^2(\mathbb{R})}^2=m\right\},
\end{equation}
where
$$E_\infty(w):=\frac{1}{2}\|\partial_z w\|_{L_z^2(\mathbb{R})}^2-\frac{1}{8\pi}\|w\|_{L_z^4(\mathbb{R})}^4.$$
We recall that the problem $\mathcal{J}_\infty(m)$ possesses a positive symmetric decreasing ground state $Q_\infty$, and it solves the Euler--Lagrange equation 
\begin{equation}\label{1D EL equation}
-\partial_z^2Q_\infty-\frac{1}{2\pi}Q_\infty^3=-\mu_\infty Q_\infty
\end{equation}
with $\mu_\infty>0$ (see \cite[Theorem 8.1.6]{Caz}). Moreover, it is unique up to phase shift and translation (see \cite{K, MS1}).

Our first main result provides the derivation of the 1d ground state from the 3d minimizer in Theorem \ref{th1} with a precise rate of convergence. 

\begin{theorem}[Dimension reduction limit to the 1d ground state]\label{h1cv}
For a sufficiently large $\omega\geq 1$, let $Q_\omega$ be a minimizer for the problem $\mathcal{J}_\omega(m)$ constructed in Theorem \ref{th1}. Then, we have
$$\|Q_\omega(x)-Q_\infty(z)\Phi_0(y)\|_{\Sigma}\lesssim\frac{1}{\sqrt{\omega}}$$
and
$$\mu_\omega=\mu_\infty+O(\omega^{-1}),$$
where $\Phi_0(y)=\frac{1}{\sqrt{\pi}}e^{-\frac{|y|^2}{2}}$ is the lowest eigenstate to the 2d Hermite operator $H_y$, and $\mu_\omega$ (resp., $\mu_\infty$) is the Lagrange multiplier in \eqref{3D EL equation} (resp., \eqref{1D EL equation}). 
\end{theorem}

\begin{remark}\label{conv remark}
 More precisely, Theorem \ref{h1cv} is broken into the following: 
\begin{enumerate}
\item (3d-to-1d estimates) We have $\|Q_\omega(x)-Q_{\omega,\parallel}(z)\Phi_0(y)\|_{L_x^2(\mathbb{R}^3)\cap \dot{\Sigma}_y}\lesssim\frac{1}{\omega}$. However, for the $z$-directional derivative norm, only a weaker convergent rate $\|\partial_z (Q_\omega(x)-Q_{\omega,\parallel}(z)\Phi_0(y))\|_{L_x^2(\mathbb{R}^3)}\lesssim\frac{1}{\sqrt{\omega}}$ is obtained (see Lemma \ref{shp}).
\item (Derivation of the 1d ground state) For the $\Phi_0$-component, we have a better bound $\|Q_{\omega,\parallel}(z)\Phi_0(y)- Q_\infty(z)\Phi_0(y)\|_{H^1_x(\R^3)}=\|Q_{\omega,\parallel}(z)- Q_\infty(z)\|_{H^1_z(\R)}\lesssim\frac{1}{\omega}$ (see \eqref{reduced convergent estimate})
\end{enumerate}
\end{remark}
\begin{remark}
Theorem \ref{h1cv} improves the previous result  \cite[Theorem 2]{BBJV} in two aspects. First, for the dimension reduction \eqref{BBJV reduction}, the $O(\frac{1}{\sqrt{\omega}})$-rate of convergence in \cite[Theorem 2]{BBJV} is improved to $O(\frac{1}{\omega})$. Secondly, the limit profile is clearly characterized as the 1D ground state $Q_\infty$. 
\end{remark}

\begin{remark}\label{point wise}
Dimension reduction also holds in high Sobolev norms (see Remark \ref{high Sobolev norm bounds}). In particular, $Q_\omega(x)\to Q_\infty(z)\Phi_0(y)$ point-wisely. 
\end{remark}

Next, using the dimension-reduction limit and  the spectral properties of the 1d ground state to \eqref{1d minimization}, we establish the uniqueness of the 3d minimizer. 
\begin{theorem}[Uniqueness]\label{uniqu}
For a sufficiently large $\omega\geq 1$, the minimizer $Q_\omega$ of the problem $\mathcal{J}_\omega(m)$, constructed in Theorem \ref{th1}, is unique.
\end{theorem}

As a direct consequence, combining Theorem \ref{th1} and \ref{uniqu} by the standard argument of Cazenave and Lions \cite{CL} together with a suitable global well-posedness (Proposition \ref{prop: global existence}), we prove the orbital stability for the 3d NLS \eqref{NLS}. 

\begin{theorem}[Orbital stability]\label{stab}
For a sufficiently large $\omega\geq 1$, let $Q_\omega$ be the unique minimizer of the variational problem $\mathcal{J}_\omega(m)$ constructed in Theorem \ref{th1}. Then, for any $\epsilon>0$, there exists $\delta>0$ such that if $\| u_0(x)-Q_\omega(x)\|_{\Sigma} \le \delta$, then the global solution $u_\omega(t)\in C_t(\mathbb{R};\Sigma)$ to the 3d NLS \eqref{NLS} with initial data $u_0$ satisfies
$$\inf_{z_1, \theta_1\in \R}   \|u_\omega(t,x)-e^{i\theta_1}Q_\omega(y,z-z_1)\|_{ {\Sigma}}  \le \epsilon\quad\textup{for all }t\in\mathbb{R}.$$
\end{theorem}

\begin{remark}
\begin{enumerate}
\item Since $\|\cdot\|_\Sigma$ is equivalent to $\|\cdot\|_{\dot{\Sigma}_y}+\|\partial_z (\cdot)\|_{L^2_x(\R^3)}+\|\cdot\|_{L^2_x(\R^3)}$ on $\Sigma$, the norm $\|\cdot\|_\Sigma$ stated in Theorem \ref{stab} can be replaced by $\|\cdot\|_{\dot{\Sigma}_y}+\|\partial_z (\cdot)\|_{L^2_x(\R^3)}+\|\cdot\|_{L^2_x(\R^3)}$.
\item By scaling, Theorem \ref{stab} corresponds to the orbital stability in \cite[Theorem 1]{BBJV}; however, by the uniqueness (Theorem \ref{uniqu}), the possibility of transition from one minimizer to another is eliminated. 
\end{enumerate}
\end{remark}

To sum up, the main contribution of this papers is to clarify the quasi-lower dimensional properties of partially confined BECs; we establish the emergence of the 1D ground state from the 3D energy minimization problem with a precise rate of convergence, but we also prove uniqueness and dynamical stability of the 3D minimizer. These results are based on the introduction of the setup \eqref{minimization}, which we think fits better for the dimension reduction. Indeed, this formulation is consistent with the setup of the mean-field limit in Chen and Holmer \cite{CH}.  
Also, the strong confinement formulation $(\omega\to\infty)$ seems to provide a clearer picture of the dimension-reduction process than the small-mass-limit formulation (see \cite{BBJV} for the small-mass-limit). 
Indeed, we note that if $\omega$ is sufficiently large, the Hermite operator $\omega (H_y-2)$ acts completely differently on the lowest and higher eigenstates. Thus, our intuition naturally leads to the modification of the Gagliardo--Nirenberg inequality (see Lemma \ref{GN type inequality}). 
In our analysis, this modified inequality plays a fundamental role in handling the anisotropic operator $\omega (H_y-2)-\partial_z^2$. For instance, it is employed to prove the conditional global existence (Proposition \ref{prop: global existence}) and some properties of the 3d minimizer $Q_\omega$, such as the concentration of the lowest eigenstate (Lemma \ref{basic uniform bound} and \ref{shp}). It is also helpful in the proof of the existence of a minimizer (Theorem \ref{th1}).

Once the dimension reduction limit of a 3d energy minimizer (Theorem \ref{h1cv}) has been justified, we use convergence to prove its uniqueness (Theorem \ref{uniqu}). A key step is to obtain a coercivity estimate of the linearized operator at the 3d minimizer $Q_\omega$ (Proposition \ref{coe}). It can be shown by transferring the coercivity of the linearized operator for the 1d ground state $Q_\infty$ via the dimension reduction limit. The non-degeneracy of the 3d linearized operator can also be proved by dimension reduction, as in \cite{Lenzmann}. Then, we show that if there are two minimizers, comparing modified energies, the coercivity bound deduces a contradiction. 

Finally, we note that our proof relies on the fact that even though the 3d minimization problem \eqref{minimization} is mass-supercritical, the limiting 1d problem \eqref{1d minimization} is mass-subcritical and has a ground state. Thus, the main results of this study can be extended to general dimensions and nonlinearities in a similar situation. However, another physically relevant 3d-to-2d reduction problem cannot be treated by the current method because the limiting 2d cubic NLS is mass-critical.

\subsection{Organization of the paper}
The remainder of this paper is organized as follows. In Section \ref{basic}, we introduce the notations used in this paper and prove some preliminary estimates. Global well-posedness is given for the trapped solutions to the 3d NLS. In Section \ref{minexis}, we prove the existence of a minimizer for $\mathcal{J}_\omega(m)$ and state some uniform bounds for a minimizer and vanishing rate of the projection of a minimizer onto higher eigenstates of the 2d Hermite operator. Section \ref{dimreduc} is devoted to the study of the convergence rate from 3d to the 1d NLS. In Section \ref{uniqsec}, we study the linearized operator for the 3d NLS and prove the uniqueness of the minimizer. The dimension reduction of the Cauchy problem is included in the Appendix for the readers' convenience. 

\subsection{Acknowledgement}
This work was supported by National Research Foundation of Korea (NRF) grant funded by the Korean government (MSIT) (No. NRF-2020R1A2C4002615)

\section{Basic tools, and global existence for trapped solutions}\label{basic}

In this section, we introduce basic analysis tools to deal with the anisotropic elliptic operator $\omega H_y-\partial_z^2$, and prove that if a negative energy solution to the time-dependent 3d NLS \eqref{NLS} initially obeys the constraint in the variational problem \eqref{minimization}, then it exists globally in time, and a refined constraint is satisfied for all times (see Proposition \ref{prop: global existence} below). We remark that proving such a conditional global well-posedness is a prerequisite for orbital stability because the standard approach by Cazenave and Lions \cite{CL} can be applied under the assumption that any perturbed state satisfies the constraint for all time.

\subsection{Spectral representation, and notations}
For the 2d Hermite operator $H_y=-\Delta_y+|y|^2$, let $\{\Phi_j\}_{j=0}^\infty\subset L_y^2(\mathbb{R}^2)$ be the collection of $L_y^2(\mathbb{R}^2)$-normalized eigenfunctions, that is, 
$$H_y \Phi_j=\Lambda_j \Phi_j,$$
with eigenvalues $\Lambda_0< \Lambda_1\leq \Lambda_2\leq \cdots$ in a non-decreasing order. We recall that $\{\Phi_j\}_{j=0}^\infty$ forms an orthonormal basis of $L_y^2(\mathbb{R}^2)$, and that the lowest eigenvalue is $\Lambda_0=2$, and the corresponding eigenstate is given by $\Phi_0(y)=\frac{1}{\sqrt{\pi}}e^{-\frac{|y|^2}{2}}$. From the spectral representation, the function $u\in L_x^2(\mathbb{R}^3)$ can be written as
\begin{equation}\label{spectral representation}
u(x)=\sum_{j=0}^\infty \langle u(\cdot,z), \Phi_j\rangle_{L_y^2(\mathbb{R}^2)}\Phi_j(y),
\end{equation}
where
$$\langle u(\cdot,z), \Phi_j\rangle_{L_y^2(\mathbb{R}^2)}=\int_{\mathbb{R}^2}u(y,z)\Phi_j(y)dy:\mathbb{R}_z\to\mathbb{C}.$$
In particular, we denote the $\Phi_0(y)$-directional component of $u=u(x):\mathbb{R}_x^3\to\mathbb{C}$ by
$$u_\parallel(z)=\langle u(\cdot,z), \Phi_0\rangle_{L_y^2(\mathbb{R}^2)}:\mathbb{R}_z\to\mathbb{C}.$$
We define the 2d projection onto the lowest eigenspace by
$$(P_0u)(x):=u_\parallel(z)\Phi_0(y):\mathbb{R}_x^3\to\mathbb{C},$$
and let $P_1=1-P_0$ be the projection to the orthogonal complement, precisely,
$$(P_1u)(x)=\sum_{j=1}^\infty \langle u(\cdot,z), \Phi_j\rangle_{L_y^2(\mathbb{R}^2)}\Phi_j(y):\mathbb{R}_x^3\to\mathbb{C}.$$
Then, we have
$$\|u\|_{L_x^2(\mathbb{R}^3)}^2=\sum_{j=0}^\infty \big\|\langle u(\cdot,z), \Phi_j\rangle_{L_y^2(\mathbb{R}^2)}\big\|_{L_z^2(\mathbb{R})}^2=\|P_0u\|_{L_x^2(\mathbb{R}^3)}^2+\|P_1u\|_{L_x^2(\mathbb{R}^3)}^2.$$
Using the spectral representation, we prove the following interpolation inequality.
\begin{lemma}[Interpolation inequality]\label{interpo}
For any $k\in\mathbb{N}$ and $\theta\in(0,1)$, we have
$$
\| (H_y  -\partial_z^2)^k u\|_{L^2_x(\R^3)}\le  \|u\|_{L^2_x(\R^3)}^{1-\theta}\| (H_y  -\partial_z^2)^\frac{k}{\theta}u\|_{L^2_x(\R^3)}^\theta.
$$
\end{lemma}
\begin{proof}
From the spectral representation \eqref{spectral representation} with $c_j(z)=\langle u(\cdot,z), \Phi_j\rangle_{L_y^2(\mathbb{R}^2)}$, we write 
$$((H_y  -\partial_z^2)^k u)(x)=\sum_{j=0}^\infty ((\Lambda_j  -\partial_z^2)^k c_j)(z) \Phi_j(y).$$
Then, applying the Plancherel theorem with respect to the $z$-variable, we obtain 
$$\| (H_y  -\partial_z^2)^k \varphi\|_{L^2_x(\R^3)}^2=\sum_{j=0}^\infty\big\|((\Lambda_j  -\partial_z^2)^k c_j)(z)\big\|_{L_z^2(\mathbb{R})}^2=\frac{1}{2\pi}\sum_{j=0}^\infty\big\|(\Lambda_j+\xi^2)^k\hat{c}_j(\xi)\big\|_{L_\xi^2(\mathbb{R})}^2.$$
Hence, by the standard interpolation inequality, it follows that 
$$\| (H_y  -\partial_z^2)^k \varphi\|_{L^2_x(\R^3)}^2\leq\left\{\frac{1}{2\pi}\sum_{j=0}^\infty\|\hat{c}_j(\xi)\|_{L_\xi^2(\mathbb{R})}^2\right\}^{1-\theta}\left\{\frac{1}{2\pi}\sum_{j=0}^\infty\big\|(\Lambda_j+\xi^2)^{\frac{k}{\theta}}\hat{c}_j(\xi)\big\|_{L_\xi^2(\mathbb{R})}^2\right\}^\theta.$$
Then, applying the Plancherel theorem and the spectral representation backward, we prove the lemma.
\end{proof} 

\subsection{Gagliardo--Nirenberg inequality}
Throughout this article, our analysis relies heavily on the following modified Gagliardo--Nirenberg inequality. It is a simple modification of the standard inequality; however, it is useful for capturing the $\omega\to\infty$ limit behavior of the anisotropic elliptic operator $\omega H_y-\partial_z^2$. 
\begin{lemma}[Gagliardo--Nirenberg inequality]\label{GN type inequality}
There exists $C_{GN}>2$ such that 
$$\begin{aligned}
\|u\|_{L_x^4(\mathbb{R}^3)}^4&\leq C_{GN}\Big\{\|P_0u\|_{L_x^2(\mathbb{R}^3)}^3\|P_0\partial_z u\|_{L_x^2(\mathbb{R}^3)}+ \|P_1u\|_{L_x^2(\mathbb{R}^3)}\|P_1\partial_z u\|_{L_x^2(\mathbb{R}^3)}\|P_1u\|_{\dot{\Sigma}_y}^2\Big\},
\end{aligned}$$
where $\|\cdot\|_{\dot{\Sigma}_y}$ is the norm defined by \eqref{energy norm}.
\end{lemma}

\begin{proof}
By the 1d Gagliardo--Nirenberg and the H\"older inequalities, we obtain
$$\begin{aligned}
\|u\|_{L_x^4(\mathbb{R}^3)}^4&=\Big\|\|u\|_{L_z^4(\mathbb{R})}\Big\|_{L_y^4(\mathbb{R}^2)}^4\lesssim\Big\|\|u\|_{L_z^2(\mathbb{R})}^{\frac{3}{4}} \|\partial_z u\|_{L_z^2(\mathbb{R})}^{\frac{1}{4}} \Big\|_{L_y^4(\mathbb{R}^2)}^4\\
&\leq\Big\|\|u\|_{L_z^2(\mathbb{R})}^{\frac{3}{4}}\Big\|_{L_y^8(\mathbb{R}^2)}^4 \Big\|\|\partial_z u\|_{L_z^2(\mathbb{R})}^{\frac{1}{4}} \Big\|_{L_y^8(\mathbb{R}^2)}^4=\Big\|\|u\|_{L_z^2(\mathbb{R})}\Big\|_{L_y^6(\mathbb{R}^2)}^3\|\partial_z  u\|_{L_x^2(\mathbb{R}^3)}\\
&\leq \Big\|\|u\|_{L_y^6(\mathbb{R}^2)}\Big\|_{L_z^2(\mathbb{R})}^3\|\partial_z  u\|_{L_x^2(\mathbb{R}^3)}.
\end{aligned}$$
Consequently, we apply the 2d Gagliardo--Nirenberg inequality, 
$$\begin{aligned}
\|u\|_{L_x^4(\mathbb{R}^3)}^4&\lesssim \Big\|\|u\|_{L_y^2(\mathbb{R}^2)}^{\frac{1}{3}}\|\nabla_yu\|_{L_y^2(\mathbb{R}^2)}^{\frac{2}{3}}\Big\|_{L_z^2(\mathbb{R})}^3\|\partial_z  u\|_{L_x^2(\mathbb{R}^3)}\\
&\leq \|u\|_{L_x^2(\mathbb{R}^3)}\|\nabla_y u\|_{L_x^2(\mathbb{R}^3)}^2\|\partial_z  u\|_{L_x^2(\mathbb{R}^3)}\\
&\leq \|u\|_{L_x^2(\mathbb{R}^3)}\|\sqrt{H_y}u\|_{L_x^2(\mathbb{R}^3)}^2\|\partial_z  u\|_{L_x^2(\mathbb{R}^3)}.
\end{aligned}$$
In the above bound, by inserting $P_0u$ and $P_1u$ with 
$$
\sqrt{H_y}P_0u=\sqrt{2}P_0u \ \textup{ and } \|\sqrt{H_y} P_1 u\|_{L_x^2(\mathbb{R}^3)}\lesssim \|P_1u\|_{\dot{\Sigma}_y},
$$ respectively, and then combining them, we complete the proof.
\end{proof}

As an application, we find a forbidden region in the weighted energy space $\Sigma$ (see \eqref{energy space}).
 
\begin{corollary}[Forbidden region]\label{cor: forbidden region}
Suppose that $\omega \ge (C_{GN})^4m^2$, where $C_{GN}$ is a constant in Lemma \ref{GN type inequality}. Then, there is no $u\in\Sigma$ such that $M(u)=m$, $E_\omega(u)<0$ and
$$\frac{(C_{GN})^2m^3}{2\omega} \leq\|u\|_{\dot{\Sigma}_y}^2\leq \sqrt{\omega}.$$
As a consequence, if $M(u)=m$, $E_\omega(u)<0$, and $\|u\|_{\dot{\Sigma}_y}^2\leq \sqrt{\omega}$, then $\|u\|_{\dot{\Sigma}_y}^2<\frac{(C_{GN})^2m^3}{2\omega}$.
\end{corollary}
\begin{proof}
For contradiction, we assume that there is a nonzero $u\in\Sigma$. Then, by Lemma \ref{GN type inequality} and the Cauchy-Schwarz inequality $ab\le \frac{a^2}{8}+2b^2$ with $a=\|\partial_z u\|_{L_x^2(\mathbb{R}^3)}$, we have
\begin{equation}\label{proof: forbidden region}
\begin{aligned}
0>E_\omega(u)&\geq\frac{\omega}{2}\|u\|_{\dot{\Sigma}_y}^2+\frac{1}{2}\|\partial_z u\|_{L_x^2(\mathbb{R}^3)}^2-\frac{C_{GN}m^{\frac{3}{2}}}{4}\|\partial_z u\|_{L_x^2(\mathbb{R}^3)}\\
&\quad-\frac{C_{GN}\sqrt{m}}{4}\|\partial_z u\|_{L_x^2(\mathbb{R}^3)}\|u\|_{\dot{\Sigma}_y}^2\\
&\geq\frac{\omega}{2}\|u\|_{\dot{\Sigma}_y}^2+\frac{1}{4}\|\partial_z u\|_{L_x^2(\mathbb{R}^3)}^2-\frac{(C_{GN})^2m^3}{8}-\frac{(C_{GN})^2m}{8}\|u\|_{\dot{\Sigma}_y}^4.
\end{aligned}
\end{equation}
By the assumption on $\|u\|_{\dot{\Sigma}_y}^2$, it follows that 
\begin{equation*}\label{proof: forbidden region1}
\begin{aligned}
0&>\frac{\omega}{2}\|u\|_{\dot{\Sigma}_y}^2+\frac{1}{4}\|\partial_z u\|_{L_x^2(\mathbb{R}^3)}^2-\frac{\omega}{4}\|u\|_{\dot{\Sigma}_y}^2-\frac{(C_{GN})^2m}{8}\sqrt{\omega}\|u\|_{\dot{\Sigma}_y}^2\\
&=\frac{1}{4}\|\partial_z u\|_{L_x^2(\mathbb{R}^3)}^2+\frac{\omega}{4}\|u\|_{\dot{\Sigma}_y}^2\bigg\{1-\frac{(C_{GN})^2m}{2\sqrt{\omega}}\bigg\}.
\end{aligned}
\end{equation*}
However, because $\omega \ge  {(C_{GN})^4m^2}$, we can deduce a contradiction.
\end{proof}

\begin{remark}
Corollary \ref{cor: forbidden region} states that if $u$ has negative energy and satisfies the constraint in the variational problem \eqref{minimization}, then Lemma \ref{GN type inequality} can be read as 
$$\|u\|_{L_x^4(\mathbb{R}^3)}^4\lesssim \|P_0u\|_{L_x^2(\mathbb{R}^3)}^3\|P_0\partial_z u\|_{L_x^2(\mathbb{R}^3)}+\frac{1}{\omega}\|P_1u\|_{L_x^2(\mathbb{R}^3)}\|P_1\partial_z u\|_{L_x^2(\mathbb{R}^3)}\lesssim \|\partial_z u\|_{L_x^2(\mathbb{R}^3)}.$$
Thus, the modified Gagliardo--Nirenberg inequality behaves like the standard 1d inequality
$$\|u\|_{L_x^4(\mathbb{R})}^4\lesssim   \|u\|_{L_x^2(\mathbb{R})}^3\|\partial_z u\|_{L_x^2(\mathbb{R})}.$$
In this sense, the modified inequality is a suitable tool for solving a supercritical problem \eqref{minimization} with a subcritical nature in the limit.
\end{remark}

\subsection{Conditional global existence}
We consider the Cauchy problem for 3d NLS \eqref{NLS}. Using Mehler's formula \cite{Mehler}, the linear Schr\"odinger flow with a partial harmonic potential satisfies the Strichartz estimates \cite{KT}. Thus, a standard fixed-point argument yields the local well-posedness of the NLS \eqref{NLS} in the weighted energy space $\Sigma$ \cite{Carles, Caz}. We note that its solution exists as long as the quantity $\|\partial_z u(t)\|_{L^2(\mathbb{R}^3)}+\|u(t)\|_{\dot{\Sigma}_y}$ remains bounded; however, a blow-up may occur in finite time.

We show that negative energy solutions whose initial data satisfy the constraint exist globally in time, and they obey a refined constraint.

\begin{proposition}[Global existence for trapped solutions]\label{prop: global existence}
Let $\omega \ge (C_{GN})^4m^2$, where $C_{GN}$ is a constant in Lemma \ref{GN type inequality}. Suppose that $u_{\omega,0}\in \Sigma$, $M(u_{\omega,0})=m$.
$$E_\omega(u_{\omega,0})<0\textup{ and }\|u_{\omega,0}\|_{\dot{\Sigma}_y}^2\leq \sqrt{\omega}.$$
Then, the solution $u_\omega(t)$ to the 3d NLS \eqref{NLS} with initial data $u_{\omega,0}$ exists globally in time, and 
\begin{equation}\label{eq: global existence}
\sup_{t\in\mathbb{R}}\|u_{\omega}(t)\|_{\dot{\Sigma}_y}^2\leq\frac{(C_{GN})^2m^3}{2\omega}.
\end{equation}
\end{proposition}

\begin{proof}
By the time reversal symmetry, it is sufficient to consider only positive times. Let $u_\omega(t)$ be the solution to the 3d NLS \eqref{NLS} with initial data $u_{\omega,0}$ on the maximal interval of existence $[0,T)$. Then, it follows from Corollary \ref{cor: forbidden region} and the continuity of the nonlinear solution that $\|u_{\omega}(t)\|_{\dot{\Sigma}_y}^2\leq \frac{(C_{GN})^2m^3}{2\omega}$ and by \eqref{proof: forbidden region}, 
$$\|\partial_z u_{\omega}(t)\|_{L_x^2(\mathbb{R}^3)}^2 \leq\frac{(C_{GN})^2m^3}{2}+\frac{(C_{GN})^2m}{2}\|u_{\omega}(t)\|_{\dot{\Sigma}_y}^4\leq\frac{(C_{GN})^2m^3}{2}+\frac{(C_{GN})^6m^7}{8\omega^2}$$
for all $0\leq t< T$. Because $u_{\omega}(t)$ remains bounded, we must have $T=\infty$.
\end{proof} 
Next, we prove that a finite-time blow-up may occur if the additional constraint is not satisfied.
\begin{lemma}\label{blow up lemma}
Let  $\omega\ge  (C_{GN})^4m^2$. If there exists $u_0\in \Sigma$ such that 
\begin{equation}\label{assump on blow up}
\|x u_0\|_{L^2_x(\R^3)}<\infty, \ \  \|u_0\|_{\dot{\Sigma}_y}^2\ge \sqrt{\omega}, \ \  M(u_0)=m \ \  \textup{ and }\ \  E_\omega(u_0)<0,
\end{equation}
 then the solution to the Cauchy problem \eqref{NLS} with initial data $u_0$ blows up in finite time.
\end{lemma}

\begin{proof}
We follow the classical argument in Glassey \cite{Glassey}. Let $u(t)$ be the solution to the 3d NLS \eqref{NLS} with initial data $u_0$ on the maximal interval of existence $(-T_{\min},T_{\max})$. Note that 
\begin{equation}\label{const outside}
\|u(t)\|_{\dot{\Sigma}_y}^2\ge \sqrt{\omega} \ \  \textup{ for } \ \ t\in (-T_{\min},T_{\max}),
\end{equation}
because Proposition \ref{prop: global existence} implies that if $\|u(t_0)\|_{\dot{\Sigma}_y}^2< \sqrt{\omega}$ for some $t_0\in (-T_{\min},T_{\max})$, then $\|u(t)\|_{\dot{\Sigma}_y}^2  \leq\frac{(C_{GN})^2m^3}{2\omega}<\sqrt{\omega}$ for all $t\in\mathbb{R}$ including $t=0$.

On the other hand, by the scaling $v(t,y,z)=\sqrt{\omega}u(t,\sqrt{\omega}y,z)$, we observe that the problem \eqref{NLS} is equivalent to 
$$i\partial_t v=(-\Delta_x+\omega^2|y|^2 -2\omega)v-\frac{1}{\omega}|v|^2v.$$
Then, direct computations (see \cite[Proposition 6.5.1]{Caz}) yield
$$\begin{aligned}
\partial_t^2 \int_{\R^3} |x|^2| v(t,x)|^2 dx&=8 \|\nabla v\|_{L^2(\R^3)}^2-6\omega^{-1}\|v\|_{L^4(\R^3)}^4-8\omega^2\int_{\R^3} |y|^2 |v|^2dx\\
&=8\omega\|u\|_{\dot{\Sigma}_y}^2+16 \omega\|u\|_{L^2_x(\R^3)}^2+8\|\partial_z u\|_{L_x^2(\mathbb{R}^3)}^2 -6 \|u\|_{L^4_x(\R^3)}^4\\
&=24 E_\omega(u)-4\omega\|u\|_{\dot{\Sigma}_y}^2-4\|\partial_z u\|_{L^2_x(\R^3)}^2 +16\omega\|u\|_{L^2_x(\R^3)}^2.
\end{aligned}$$
Hence, by \eqref{const outside}, the assumptions \eqref{assump on blow up} and the fact that $C_{GN}>2$ (see Lemma \ref{GN type inequality}), if $\omega\ge  (C_{GN})^4m^2$, we obtain
$$\partial_t^2 \int_{\R^3} |x|^2| v(t,x)|^2 dx\le 24 E_\omega(u_0)<0.$$
Therefore, we conclude that both $T_{\min}$ and $T_{\max}$ are finite.
\end{proof}

\section{Existence of a minimizer: proof of Theorem \ref{th1}}\label{minexis}

We consider the minimization problem $\mathcal{J}_\omega(m)$ (see \eqref{minimization}). This section proves the existence of a minimizer.   Indeed, as mentioned in Remark \ref{remark: existence remark} (1), existence has been established in a slightly different setting \cite[Theorem 1]{BBJV}; however, it is reformulated in the context of dimension reduction as $\omega\to\infty$.  Thus, we only sketch the proof for the sake of completeness. We also provide additional properties of a minimizer, which are direct consequences of this different formulation (Lemmas \ref{hig} and \ref{shp}). 

\subsection{Existence of a minimizer}

First, we obtain the upper and lower bounds of the minimum energy level of the variational problem \eqref{minimization}.

\begin{lemma}\label{upp}
For any $\omega> 0$, we have 
$$-\infty<\mathcal{J}_\omega(m)\leq \mathcal{J}_\infty(m)=E_\infty(Q_\infty)<0.$$
\end{lemma}

\begin{proof}
Direct calculations show that $ \|\Phi_0(y)Q_\infty(z)\|_{\dot{\Sigma}_y}=0$
and
$$E_\omega(\Phi_0(y)Q_\infty(z))=\frac{1}{2}\|\partial_z Q_\infty\|_{L_z^2(\mathbb{R})}^2-\frac{1}{4}\|Q_\infty\|_{L_z^4(\mathbb{R})}^4\|\Phi_0\|_{L^4(\mathbb{R}^2)}^4=E_\infty(Q_\infty)=\mathcal{J}_\infty(m)<0,$$
since
\begin{equation}\label{L^4 norm of Phi_0}
\int_{\mathbb{R}^2}|\Phi_0(y)|^4dy=\int_{\mathbb{R}^2}\frac{1}{\pi^2}e^{-2|y|^2}dy=\frac{1}{2\pi}.
\end{equation}
Thus, by minimality, it follows that $\mathcal{J}_\omega(m)\leq\mathcal{J}_\infty(m)$.

To show that $\mathcal{J}_\omega(m)$ is bounded from below, we observe that by Lemma \ref{GN type inequality} and the mass constraint, we have 
$$E_\omega(u)\geq \frac{\omega}{2}\|u\|_{\dot{\Sigma}_y}^2+\frac{1}{2}\|\partial_z u\|_{L_x^2(\mathbb{R}^3)}^2 -\frac{C_{GN}}{4}\Big\{m^{\frac{3}{2}}\|\partial_z u\|_{L_x^2(\mathbb{R}^3)}+ m^\frac12\|\partial_z u\|_{L_x^2(\mathbb{R}^3)}\sqrt{\omega} \Big\}.$$
Thus, the Cauchy-Schwarz inequality yields a lower bound on the energy $E_\omega(u)$, which is independent of $u$.
\end{proof}

Now, we prove the existence of a minimizer.

\begin{proof}[Sketch of the proof of Theorem \ref{th1}]
Let $\{u_n\}_{n=1}^\infty$ be a minimizing sequence for $\mathcal{J}_\omega(m)$. By Lemma \ref{upp}, estimate \eqref{proof: forbidden region}, and the assumption $\omega \ge {(C_{GN})^4m^2}$, we have
\begin{equation}\label{lba} 
\begin{aligned}
0&>\mathcal{J}_\omega(m)+o_n(1)=E_\omega(u_n)\\
& \geq  \frac{\omega}{2}\|u_n\|_{\dot{\Sigma}_y}^2\left(1-\frac{(C_{GN})^2m}{4\omega}\|u_n\|_{\dot{\Sigma}_y}^2\right)+\frac{1}{4}\|\partial_z u_n\|_{L_x^2(\mathbb{R}^3)}^2-\frac{(C_{GN})^2m^3}{8}\\
&\geq   \frac{\omega}{4}\|u_n\|_{\dot{\Sigma}_y}^2 +\frac{1}{4}\|\partial_z u_n\|_{L_x^2(\mathbb{R}^3)}^2-\frac{(C_{GN})^2m^3}{8}.
\end{aligned}
\end{equation}
Consequently, $\{u_n\}_{n=1}^\infty$ is bounded in $\Sigma$. Note that $\|u_n\|_{L_x^4(\mathbb{R}^3)}^4\geq -4\mathcal{J}_\infty(m)+o_n(1)\geq -2\mathcal{J}_\infty(m)>0$ because Lemma \ref{upp} implies that 
\begin{equation}\label{nont}
0>\mathcal{J}_\infty(m)\geq\limsup_{n\rightarrow \infty}E_\omega(u_n)\ge  -\frac{1}{4}\liminf_{n\rightarrow \infty}\|u_n\|_{L_x^4(\mathbb{R}^3)}^4.
\end{equation}
Hence, passing to a subsequence, there exists a sequence $\{z_n\}_{n=1}^\infty\subset \R$ such that $u_n(y,z-z_n)\rightharpoonup u_\infty$ in $\Sigma$. Thus, it follows that $u_\infty\neq 0$ (refer to the proof of \cite[Lemma 3.4]{BBJV} for details).

 With abuse of notation, we denote the function $u_n(y,z-z_n)$ by $u_n$. Then, we claim that $u_n\to u_\infty$ in $L_x^2(\mathbb{R}^3)$. Indeed, if the claim is not true, then passing to a subsequence,
$$\|u_\infty\|_{L_x^2(\mathbb{R}^3)}^2=m',\quad\lim_{n\to\infty}\|u_n-u_\infty\|_{L_x^2(\mathbb{R}^3)}^2=m-m'\in (0,m).$$
Note that since $u_n\rightharpoonup u_\infty$ in $\Sigma$, we have
$$\begin{aligned}
&\omega\|u_n\|_{\dot{\Sigma}_y}^2+\|\partial_z u_n\|_{L^2_x(\R^3)}^2\\
&=\omega\Big\{\|u_\infty\|_{\dot{\Sigma}_y}^2+\|u_n-u_\infty\|_{\dot{\Sigma}_y}^2\Big\} +\|\partial_z u_\infty\|_{L^2_x(\R^3)}^2+\|\partial_z (u_n-u_\infty)\|_{L^2_x(\R^3)}^2+o_n(1)
\end{aligned}$$
and
$$\|u_n\|_{L_x^4(\R^3)}^4=\|u_\infty\|_{L_x^4(\R^3)}^4+ \|u_n-u_\infty\|_{L_x^4(\R^3)}^4+o_n(1).$$
Hence, it follows that 
\begin{equation}\label{e1}
\begin{aligned}
\mathcal{J}_\omega(m)&=E_\omega(u_n)+o_n(1)=E_\omega(u_\infty)+E_\omega(u_n-u_\infty)+o_n(1)\\
&\geq \mathcal{J}_\omega(m')+\mathcal{J}_\omega(m-m')+o_n(1).
\end{aligned}
\end{equation}
Let $\{v_n\}_{n=1}^\infty$ be a minimizing sequence for $\mathcal{J}_\omega(m')$. Then, the modified sequence $\{\sqrt{\frac{m}{m'}}v_n\}_{n=1}^\infty$ satisfies $\|\sqrt{\frac{m}{m'}}v_n\|_{L_x^2(\mathbb{R}^3)}^2=m$ and
$$\left\| \sqrt{\tfrac{m}{m'}}v_n \right\|_{\dot{\Sigma}_y}^2\leq \frac{m}{m'}\|v_n\|_{\dot{\Sigma}_y}^2\leq  \frac{m}{m'}\frac{(C_{GN})^2(m^\prime)^3}{2\omega}\leq \sqrt{\omega},$$
where we used Corollary \ref{cor: forbidden region} and the assumption $\omega \ge  (C_{GN})^4m^2 \ge  (C_{GN})^\frac43m^2$. Moreover, by repeating the proof of \eqref{nont}, we can show that $\frac{1}{4}\|v_n\|_{L_x^4(\mathbb{R}^3)}^4\geq-\mathcal{J}_\infty(m')+o_n(1)$. Thus, by minimality, it follows that 
$$\begin{aligned}
\mathcal{J}_\omega(m)\leq E_\omega\left(\sqrt{\tfrac{m}{m'}}v_n\right)&=\frac{\omega}{2}\frac{m}{m'}\|v_n\|_{\dot{\Sigma}_y}^2+\frac{1}{2}\frac{m}{m'}\|\partial_z v_n\|_{L_x^2(\mathbb{R}^3)}^2-\frac{m^2}{4(m')^2}\|v_n\|_{L_x^4(\mathbb{R}^3)}^4\\
&=\frac{m}{m'}E_\omega(v_n)-\frac{m}{4m'}\frac{m-m'}{m'}\|v_n\|_{L_x^4(\mathbb{R}^3)}^4\\
&=\frac{m}{m'}\mathcal{J}_\omega(m')+\frac{m}{m'}\frac{m-m'}{m'}\mathcal{J}_\infty(m')+o_n(1),
\end{aligned}$$
and thus
$$\mathcal{J}_\omega(m')\ge \frac{m^\prime}{m}\mathcal{J}_\omega(m)-\frac{m-m'}{m'}\mathcal{J}_\infty(m')+o_n(1).$$
On the other hand, by switching the roles of $m^\prime$ and $m-m^\prime$, one can show that 
$$\mathcal{J}_\omega(m-m')\ge \frac{m-m^\prime}{m}\mathcal{J}_\omega(m)-\frac{m'}{m-m'}\mathcal{J}_\infty(m-m')+o_n(1).$$
Then, inserting these two lower bounds in \eqref{e1}, we obtain 
\begin{align*}
\mathcal{J}_\omega(m)&\geq \mathcal{J}_\omega(m')+\mathcal{J}_\omega(m-m')+o_n(1)\\
&\ge \mathcal{J}_\omega(m)-\frac{m-m'}{m'}\mathcal{J}_\infty(m')-\frac{m'}{m-m'}\mathcal{J}_\infty(m-m')+o_n(1),
\end{align*}
which deduces a contradiction with $\mathcal{J}_\infty(m'), \mathcal{J}_\infty(m-m')<0$ (see Lemma \ref{upp}). Thus, we conclude that $u_n\to u_\infty$ in $L_x^2(\mathbb{R}^3)$ and $\|u_\infty\|_{L_x^2(\mathbb{R}^3)}^2=m$.

We claim that $u_n\rightarrow u_\infty$ in $\Sigma$ and $u_\infty$ is a minimizer for $\mathcal{J}_\omega(m)$. The $L_x^2(\mathbb{R}^3)$ convergence $u_n\to u_\infty$ and the Gagliardo--Nirenberg inequality $\|u\|_{L_x^4(\R^3)}^4\lesssim\|u\|_{L_x^2(\R^3)}\|u\|_{H_x^1(\R^3)}^3$ yield the convergence $u_n\to u_\infty$ in $L_x^4(\mathbb{R}^3)$ and $E_\omega(u_n-u_\infty)\geq \|u_n- u_\infty\|_\Sigma^2+ o_n(1)$. Thus, it follows from the argument used to derive \eqref{e1} that $\mathcal{J}_\omega(m)=E_\omega(u_n)+o_n(1)=E_\omega(u_\infty)+E_\omega(u_n-u_\infty)+o_n(1)\ge \mathcal{J}_\omega(m)+\|u_n- u_\infty\|_\Sigma^2+o_n(1)$.

Because the minimizer $u_\infty$ is of the form $u_\infty(x)=u_\infty(|y|, |z|)$ up to translation and phase shift, and it is non-negative and decreases with respect to $y$ and $z$, that is, $u_\infty=e^{i\theta}Q_\omega(|y|,|z-z_0|)$ for some $\theta, z_0\in \R$, we refer to \cite[Theorem 2]{BBJV}.
\end{proof}

\subsection{Uniform bounds, and vanishing higher eigenstates}
Because a minimizing sequence converges (passing to a subsequence and up to symmetries), taking $n\to \infty$ in the estimate \eqref{lba}, we obtain the following preliminary bound:

\begin{lemma}[Preliminary uniform bound]\label{basic uniform bound}
For $\omega \ge (C_{GN})^4m^2$, let $Q_\omega$ be the minimizer constructed in Theorem \ref{th1}. Then, $ \omega \|Q_\omega\|_{\dot{\Sigma}_y}^2+   \|\partial_z Q_\omega\|_{L_x^2(\mathbb{R}^3)}^2$ is uniformly bounded in $\omega$.
\end{lemma}
\begin{proof}
Following the arguments in \eqref{lba}, we have
\begin{equation*}
\begin{aligned}
0 >\mathcal{J}_\omega(m)=E_\omega( Q_\omega) \geq   \frac{\omega}{4}\|Q_\omega\|_{\dot{\Sigma}_y}^2 +\frac{1}{4}\|\partial_z Q_\omega\|_{L_x^2(\mathbb{R}^3)}^2-\frac{(C_{GN})^2m^3}{8}.
\end{aligned}
\end{equation*}
\end{proof}

We upgrade the above bound using the Euler--Lagrange equation \eqref{3D EL equation}. First, using a standard iterative argument for elliptic regularity, we prove weighted high Sobolev norm bounds.

\begin{lemma}[Weighted high Sobolev norm bounds]\label{hig}
For $\omega \ge (C_{GN})^4m^2$, let $Q_\omega$ be the minimizer constructed in Theorem \ref{th1}. Then,  for any $k\in \mathbb{N}$, we have
$$\sup_{\omega \ge (C_{GN})^4m^2}\big\| (H_y  -\partial_z^2)^k Q_\omega\big\|_{L^2_x(\R^3)}<\infty.$$
In particular, $Q_\omega$ is uniformly bounded in $L_x^\infty(\mathbb{R}^3)$.
\end{lemma}

\begin{proof}
By using \eqref{3D EL equation}, the energy of the minimizer can be written as 
$$
E_\omega(Q_\omega)=-\frac{\mu_\omega}{2}\|Q_\omega\|_{L^2_x(\R^3)}^2+\frac{1}{4}\|Q_\omega\|_{L^4_x(\R^3)}^4\geq -\frac{\mu_\omega m}{2}.
$$ However, because the minimum energy is negative (Lemma \ref{upp}), we have
\begin{equation}\label{mmu}
\mu_\omega\geq 0.
\end{equation}
For $k=1$, we decompose 
$$
(H_y  -\partial_z^2) Q_\omega=(H_y  -\partial_z^2) (P_0Q_\omega)+(H_y  -\partial_z^2)(P_1Q_\omega)=2P_0Q_\omega -\partial_z^2(P_0Q_\omega)+(H_y  -\partial_z^2)(P_1Q_\omega)
$$ and estimate 
$$\begin{aligned}
\big\|(H_y  -\partial_z^2) Q_\omega\big\|_{L^2_x(\R^3)}&\leq2\|P_0Q_\omega\|_{L^2_x(\R^3)}+\|\partial_z^2(P_0Q_\omega)\|_{L^2_x(\R^3)}+\big\|(H_y  -\partial_z^2) P_1Q_\omega\big\|_{L^2_x(\R^3)}\\
&\lesssim \|Q_\omega\|_{L^2_x(\R^3)}+ \big\|\left(\omega( H_y-2)  -\partial_z^2+\mu_\omega\right) Q_\omega\big\|_{L^2_x(\R^3)},
\end{aligned}$$
where we used the spectrum gap
$$
\big\|H_y (P_1Q_\omega)\big\|_{L^2_x(\R^3)}\lesssim \big\|(H_y-2) (P_1Q_\omega)\big\|_{L^2_x(\R^3)}.
$$
Then, using the equation \eqref{3D EL equation} and Lemma \ref{basic uniform bound}, we prove that 
$$\begin{aligned}
\big\|(H_y  -\partial_z^2) Q_\omega\big\|_{L^2_x(\R^3)}&\lesssim \|Q_\omega\|_{L^2_x(\R^3)}+\|Q_\omega^3\|_{L^2_x(\R^3)}=\|Q_\omega\|_{L^2_x(\R^3)}+\|Q_\omega\|_{L^6_x(\R^3)}^3\\
&\lesssim 1+\|Q_\omega\|_{H^1_x(\R^3)}^3\lesssim 1,
\end{aligned}$$
where the implicit constants are independent of $\omega$.

For $k=2$, by repeating the above estimates with equation \eqref{3D EL equation} and the commutative properties
$$
(H_y  -\partial_z^2)  (P_1Q_\omega)=P_1\big((H_y  -\partial_z^2)Q_\omega\big) \textup{ and }   H_y  (H_y  -\partial_z^2)=  (H_y  -\partial_z^2)H_y,
$$
 we write 
$$\begin{aligned}
\big\|(H_y  -\partial_z^2)^2 Q_\omega\big\|_{L^2_x(\R^3)}&=\big\|(H_y  -\partial_z^2)^2 P_0Q_\omega\big\|_{L^2_x(\R^3)}+\big\|(H_y  -\partial_z^2)^2 P_1Q_\omega\big\|_{L^2_x(\R^3)}\\
&\lesssim \|Q_\omega\|_{L^2_x(\R^3)}+\big\|(H_y  -\partial_z^2)(\omega(H_y  -2 )-\partial_z^2+\mu_\omega)Q_\omega\big\|_{L^2_x(\R^3)}\\
&=\|Q_\omega\|_{L^2_x(\R^3)}+\big\|(H_y  -\partial_z^2)(Q_\omega^3)\big\|_{L^2_x(\R^3)}.\end{aligned}$$ 
Then, by distributing derivatives in $(H_y  -\partial_z^2)(Q_\omega^3)$ and using the Sobolev embedding $H^2(\mathbb{R}^3)\hookrightarrow L^\infty(\mathbb{R}^3)$, we can obtain a uniform bound on $\|(H_y  -\partial_z^2)^2 Q_\omega\|_{L^2_x(\R^3)}$ using the uniform bound in the previous step. Proceeding inductively, we deduce the lemma for all $k\geq 2$.
\end{proof}

Next, we sharpen the bound in Lemma \ref{basic uniform bound} and prove the convergence of the minimum energy and the Lagrange multiplier.
\begin{lemma}\label{shp}
For $\omega \ge (C_{GN})^4m^2$, let $Q_\omega$ be the minimizer constructed in Theorem \ref{th1}.
\begin{enumerate}
\item (Minimum energy convergence)
$$\mathcal{J}_\omega(m)=\mathcal{J}_\infty(m)+O(\omega^{-1}).$$
\item (Vanishing higher eigenstates) 
$$
\|P_1Q_\omega\|_{L_x^2(\mathbb{R}^3)}\lesssim\|Q_\omega\|_{\dot{\Sigma}_y}\lesssim\frac{1}{\omega}\quad\textup{and}\quad \|\partial_z (P_1Q_\omega)\|_{L_x^2(\mathbb{R}^3)}\lesssim\frac{1}{\sqrt{\omega}}.$$
\item (Lagrange multiplier convergence)
$$\mu_\omega=\mu_\infty+O(\omega^{-1}).$$
\end{enumerate}
\end{lemma}  

\begin{proof}
First, we show that 
\begin{equation}\label{co0}
\|P_1Q_\omega\|_{L^2_x(\R^3)}\lesssim\|Q_\omega\|_{\dot{\Sigma}_y}\lesssim\frac{1}{\omega}.
\end{equation}
Because $\Lambda_1>2$ and $\mu_\omega\geq0$ (see \eqref{mmu}), we have
$$\begin{aligned}
\|P_1Q_\omega\|_{L^2_x(\R^3)}&\lesssim \| (H_y  -2 ) P_1Q_\omega\|_{L^2_x(\R^3)}=\| (H_y  -2 ) Q_\omega\|_{L^2_x(\R^3)}\\
&\le  \frac{1}{\omega}\|(\omega(H_y  -2 ) -\partial_z^2+\mu_\omega) Q_\omega\|_{L^2_x(\R^3)}=\frac{1}{\omega}\|Q_\omega^3\|_{L^2_x(\R^3)},
\end{aligned}$$
where equation \eqref{3D EL equation} is used in the last step. Then, by Lemma \ref{hig}, we obtain the desired bound \eqref{co0}.

For the remainder of the proof, we compare the energies of the two minimizers $Q_\omega(x)$ and $Q_\infty(z)$. For the 1d energy, we consider the $\Phi_0(y)$-directional component of $Q_\omega$, that is, $Q_{\omega,\parallel}(z)=\langle Q_\omega(\cdot,z),\Phi_0\rangle_{L_y^2(\mathbb{R}^2)}$. Note that by \eqref{co0}, its mass $m_\omega=\| Q_{\omega,\parallel}\|_{L_z^2(\mathbb{R})}^2$ satisfies $m_\omega=m+O(\frac{1}{\omega^2})$. Moreover, by Lemmas \ref{hig} and \eqref{co0},
we have $$
\| Q_\omega\|_{L^4_x(\R^3)}^4- \| P_0Q_\omega\|_{L^4_x(\R^3)}^4\lesssim\int_{\R^3}\left(Q_\omega^3+|P_0Q_\omega|^3\right)|P_1 Q_\omega|dx\lesssim\|P_1Q_\omega\|_{L_x^2(\mathbb{R}^3)}\lesssim\frac{1}{\omega}.
$$ 
Hence, the admissible function $\sqrt{\tfrac{m}{m_\omega}}Q_{\omega,\parallel}$ for the problem $\mathcal{J}_\infty(m)$ satisfies 
$$\begin{aligned}
E_\infty(Q_\infty)&\le E_\infty( \sqrt{\tfrac{m}{m_\omega}}Q_{\omega,\parallel})= E_\infty( Q_{\omega,\parallel})+O(\omega^{-2})\\
&=\frac12 \|\partial_z (P_0Q_\omega)\|_{L^2_x(\R^3)}^2-\frac14 \|P_0Q_\omega\|_{L^4_x(\R^3)}^4+O(\omega^{-2})\\
&=E_\omega(Q_\omega)-\frac{\omega}{2}\|Q_\omega\|_{\dot{\Sigma}_y}^2-\frac12\|\partial_z (P_1Q_\omega)\|_{L^2_x(\R^3)}^2+O(\omega^{-1}).
\end{aligned}$$
Since $E_\omega(Q_\omega)\leq E_\infty(Q_\infty)$, it follows that 
\begin{equation}\label{co4}
 \omega \|Q_\omega\|_{\dot{\Sigma}_y}^2+\|\partial_z(P_1Q_\omega)\|_{L^2_x(\R^3)}^2\lesssim\frac{1}{\omega}.
\end{equation}
Then, inserting this back, we obtain the minimum energy convergence $\mathcal{J}_\omega(m)=\mathcal{J}_\infty(m)+O(\omega^{-1})$. 
\begin{equation}\label{co1}
E_\omega(Q_\omega)\leq E_\infty(Q_\infty)\leq E_\infty( \sqrt{\tfrac{m}{m_\omega}}Q_{\omega,\parallel})= E_\omega(Q_\omega)+O(\omega^{-1}).
\end{equation}

It remains to show the convergence of the Lagrange multiplier. To estimate the difference between $\mu_\omega$ and $\mu_\infty$, we express the energies of the two minimizers using the Pohozaev-type identities. Here, the trick is to take the inner product with $z\partial_zQ_\omega$ rather than the usual choice $x\cdot\nabla_x Q_\omega$ because the $z$-direction is dominant in the limit. Specifically, by multiplying the elliptic equation \eqref{3D EL equation} by $Q_\omega$ and $z\partial_z Q_\omega$ and integrating over $\R^3_x$, we have
\begin{align*}
0&=\langle\omega(H_y  -2 )Q_\omega-\partial_z^2Q_\omega- Q_\omega^3+\mu_\omega Q_\omega, Q_\omega \rangle_{L^2_x(\R^3)}\\
&=\omega\|Q_\omega\|_{\dot{\Sigma}_y}^2+ \|\partial_zQ_\omega\|_{L^2_x(\R^3)}^2- \| Q_\omega\|_{L^4_x(\R^3)}^4+\mu_\omega m
\end{align*}
and
\begin{align*}
0&=\langle\omega(H_y  -2 )Q_\omega-\partial_z^2Q_\omega- Q_\omega^3+\mu_\omega Q_\omega,z\partial_z Q_\omega \rangle_{L^2_x(\R^3)}\\
&=\int_{\R^3}\frac{\omega z}{2}\partial_z(\sqrt{H_y-2}Q_\omega)^2-\frac{z}{2}\partial_z(\partial_zQ_\omega)^2-\frac{z}{4}\partial_z(Q_\omega^4)+\frac{\mu_\omega z}{2}\partial_z(Q_\omega)^2dx\\
&=-\frac{\omega}{2}\|Q_\omega\|_{\dot{\Sigma}_y}^2+\frac12\|\partial_zQ_\omega\|_{L^2_x(\R^3)}^2+\frac14\| Q_\omega\|_{L^4_x(\R^3)}^4-\frac{\mu_\omega}{2}m.
\end{align*}
Solving the above system of equations for $\|\partial_zQ_\omega\|_{L^2_x(\R^3)}^2$ and $\| Q_\omega\|_{L^4_x(\R^3)}^4$ and substituting them into the energy, we obtain 
$$\mathcal{J}_\omega(m)=E_\omega(Q_\omega)=-\frac{\mu_\omega m}{6}-\frac{\omega}{6}\|Q_\omega\|_{\dot{\Sigma}_y}^2=-\frac{\mu_\omega m}{6}+O(\omega^{-1}),$$
where \eqref{co4} is used in the last step. Similarly, we can express the energy level $E_\infty(Q_\infty)$ in terms of the mass $m$ and a Lagrange multiplier $\mu_\infty$, that is, 
$$\mathcal{J}_\infty(m)=E_\infty(Q_\omega)=-\frac{\mu_\infty m}{6}.$$
Thus, it follows from the minimum energy convergence that $\mu_\omega=\mu_\infty+O(\omega^{-1})$.
\end{proof}

\section{Dimension reduction to the 1d ground state: Proof of Theorem \ref{h1cv}}\label{dimreduc}

In this section, we prove the convergence from the 3d to 1d energy minimizers. For the proof, a key ingredient is the non-degeneracy estimate of the 1d linearized operator: 
\begin{equation}\label{1d linearized operator}
\mathcal{L}_\infty=-\partial_z^2+\mu_\infty -\frac{3}{2\pi}Q_\infty^2.
\end{equation}

\begin{lemma}[Non-degeneracy estimate \cite{K, W}]\label{non-degeneracy estimate}
The linearized operator $\mathcal{L}_\infty$ satisfies 
$\| \mathcal{L}_\infty \varphi\|_{H_z^{-1}(\mathbb{R})}\gtrsim\|\varphi\|_{H^1_z(\R)}$ for all even $\varphi \in H^1_z(\R)$.
\end{lemma}
\begin{proof}
If there exists an even function $\varphi_n\in H^1_z(\R)$ such that 
$$\|\varphi_n\|_{H^1_z(\R)}=1 \mbox{ and } \lim_{n\rightarrow \infty}\| \mathcal{L}_\infty \varphi_n\|_{H_z^{-1}(\mathbb{R})}=0,
$$
Then, we may assume that $\varphi_n\rightharpoonup \varphi_\infty$ in $H^1(\R^3)$ and $\varphi_n\rightarrow \varphi_\infty$ in $L^2_{loc}(\R)$ as $n\rightarrow \infty$. Because $\lim_{n\rightarrow \infty}\| \mathcal{L}_\infty \varphi_n\|_{H_z^{-1}(\mathbb{R})}=0$, we see that for $\psi\in C_0^\infty(\R)$,
$$
\lim_{n\rightarrow \infty}\langle \mathcal{L}_\infty \varphi_n, \psi\rangle_{L^2_z(\R)}=\langle \mathcal{L}_\infty \varphi_\infty, \psi\rangle_{L^2_z(\R)}=0,
$$
which implies that $\mathcal{L}_\infty \varphi_\infty=0$.
Moreover, by the exponential decay property of $Q_\infty$,
we see that
\begin{align*}
0&=\lim_{n\rightarrow \infty}\langle \mathcal{L}_\infty \varphi_n,\varphi_n\rangle_{L^2_z(\R)}=\lim_{n\rightarrow \infty}\int_{\R}(\partial_z\varphi_n)^2+\mu_\infty \varphi_n^2-Q_\infty^2\varphi_n^2dz\\
&=\lim_{n\rightarrow \infty}\int_{\R}(\partial_z\varphi_n)^2+\mu_\infty \varphi_n^2dz-\int_{\R}Q_\infty \varphi_\infty^2dz.
\end{align*}
Hence, we deduce that an even function $\varphi_\infty\neq0$ satisfies $\mathcal{L}_\infty \varphi_\infty=0$, which contradicts the results in \cite{K} and  \cite[Proposition 2.8]{W}.
\end{proof}

\begin{proof}[Proof of Theorem \ref{h1cv}]
By Lemma \ref{shp}, it suffices to show that 
\begin{equation}\label{reduced convergent estimate}
\|Q_{\omega,\parallel}(z)- Q_\infty(z)\|_{H^1_z(\R)}\lesssim\frac{1}{\omega}.
\end{equation}
It has been shown in the proof of Lemma \ref{shp} that $\{\sqrt{\tfrac{m}{m_\omega}}Q_{\omega,\parallel}\}_{\omega\geq\omega_0}$, where $Q_{\omega,\parallel}(z)=\langle Q_\omega(\cdot,z),\Phi_0\rangle_{L_y^2(\mathbb{R}^2)}$ is a minimizing sequence for the variational problem $\mathcal{J}_\infty(m)$ (see \eqref{co1}). Then, by the well-known variational property of $\mathcal{J}_\infty(m)$ \cite{Lions} and the uniqueness of the minimizer $Q_\infty$ \cite{K}, it follows that $Q_{\omega,\parallel}\to Q_\infty$ in $H^1_z(\R)$.

For the rate of convergence in \eqref{reduced convergent estimate}, using the Euler--Lagrange equation \eqref{3D EL equation}, we write the equation for $Q_{\omega,\parallel}$ as
$$
(-\partial_z^2+\mu_\omega)Q_{\omega,\parallel} =   \langle Q_\omega(\cdot,z)^3,\Phi_0(\cdot)\rangle_{L^2_y(\R^2)}.
$$
Then, the difference $r_{\omega,\parallel}=Q_{\omega,\parallel}- Q_\infty$ satisfies
\begin{equation}\label{uy1}
\begin{aligned}
\mathcal{L}_\infty r_{\omega,\parallel}&=(\mu_\infty-\mu_\omega)Q_{\omega,\parallel}+\frac{1}{2\pi}\left\{Q_{\omega,\parallel}^3-Q_\infty^3-3Q_\infty^2r_{\omega,\parallel}\right\}\\
&\quad +\left\{\langle Q_\omega(\cdot,z)^3,\Phi_0\rangle_{L^2_y(\R^2)}-\frac{1}{2\pi}Q_{\omega,\parallel}^3\right\},
\end{aligned}
\end{equation}
where $\mathcal{L}_\infty=-\partial_z^2+\mu_\infty -\frac{3}{2\pi}Q_\infty^2$. On the right-hand side of \eqref{uy1}, from Lemma \ref{shp}, we have that $\|(\mu_\infty-\mu_\omega)Q_{\omega,\parallel}\|_{L_x^2(\mathbb{R}^3)}\leq |\mu_\infty-\mu_\omega|\|Q_{\omega,\parallel}\|_{L_x^2(\mathbb{R}^3)}\lesssim\frac{1}{\omega}$. By the convergence $r_{\omega,\parallel}\to 0$ and the uniform bound (Lemma \ref{hig}), we have
$$\|Q_{\omega,\parallel}^3-Q_\infty^3-3Q_\infty^2r_{\omega,\parallel}\|_{L^2_z(\R)}=\|3Q_\infty r_{\omega,\parallel}^2+r_{\omega,\parallel}^3\|_{L^2_z(\R)}=o_\omega(1)\|r_{\omega,\parallel}\|_{H^1_z(\R)}.$$
Moreover, using $\frac{1}{2\pi}Q_{\omega,\parallel}^3=\langle (Q_{\omega,\parallel}(z)\Phi_0(y))^3, \Phi_0(y)\rangle_{L_y^2(\mathbb{R}^2)}=\langle (P_0Q_\omega(\cdot,z))^3, \Phi_0\rangle_{L_y^2(\mathbb{R}^2)}$ (see \eqref{L^4 norm of Phi_0}), we obtain
$$\begin{aligned}
\left\langle Q_\omega^3,\Phi_0\right\rangle_{L^2_y(\R^2)}-\frac{1}{2\pi}Q_{\omega,\parallel}^3&=\big\langle Q_\omega^3-(P_0Q_\omega)^3, \Phi_0\big\rangle_{L_y^2(\mathbb{R}^2}\\
&=\big\langle 3(P_0Q_\omega)^2(P_1Q_\omega)+3(P_0Q_\omega)(P_1Q_\omega)^2+(P_1Q_\omega)^3,\Phi_0\big\rangle_{L^2_y(\R^2)}.
\end{aligned}$$
Then, Lemma \ref{shp} with the uniform bound (Lemma \ref{hig}) yields 
$$\begin{aligned}
\left\|\langle Q_\omega(\cdot,z)^3,\Phi_0\rangle_{L^2_y(\R^2)}-\frac{1}{2\pi}Q_{\omega,\parallel}^3\right\|_{L^2_z(\R)}\lesssim\|P_1Q_\omega\|_{L^2_x(\R^3)}\lesssim\frac{1}{\omega}.
\end{aligned}$$
Putting it all together, we obtain $\|\mathcal{L}_\infty r_{\omega,\parallel}\|_{L_z^2(\mathbb{R})}\lesssim\frac{1}{\omega}$. Finally, by applying the non-degeneracy estimate for the linearized operator (Lemma \ref{non-degeneracy estimate}), we complete the proof.
\end{proof}
 
\begin{remark}[Weighted high Sobolev norm convergence]\label{high Sobolev norm bounds}
Let $k\in \mathbb{N}$. By interpolating the $L_x^2(\mathbb{R}^3)$ convergence in Theorem \ref{h1cv} and the bound in Lemma \ref{hig} using Lemma \ref{interpo}, we obtain the convergence in the high Sobolev norms:
$$\begin{aligned}
&\big\| (H_y  -\partial_z^2)^k(Q_\omega(x)-Q_\infty(z)\Phi_0(y))\big\|_{L^2_x(\R^3)}\\
&\le  \|Q_\omega(x)-Q_\infty(z)\Phi_0(y) \|_{L^2_x(\R^3)}^{1-\eta}\big\| (H_y  -\partial_z^2)^\frac{k}{\eta}(Q_\omega(x)-Q_\infty(z)\Phi_0(y))\big\|_{L^2_x(\R^3)}^\eta\\
&\leq C_{k,\eta}\omega^{-(1-\eta)}\to 0,
\end{aligned}$$
where $C_{k,\eta}>0$ is a constant depending on $k$ and $\eta$, and $\eta\in (0,1)$ satisfying $k \eta^{-1}\in \mathbb{N}$.
\end{remark}

\section{Linearized operator and uniqueness of a minimizer: Proof of Theorem \ref{uniqu}}\label{uniqsec}

In this section, we study the linearized operator at an energy minimizer of the variational problem $\mathcal{J}_\omega(m)$. Then, exploiting its coercivity, we establish the uniqueness of the minimizer, which is the main result of this study.

\subsection{Linearized operator}

Let $Q_\infty$ be the unique radially symmetric positive ground state of the 1d minimization problem (see \eqref{1d minimization}). It is well-known that $Q_\infty$ is \textit{non-degenerate} in the sense that the kernel of the 1d linearized operator (see \eqref{1d linearized operator}) acting on $L_z^2(\mathbb{R})$ with the domain $H_z^2(\mathbb{R})$ is completely characterized by the translation invariance of equation \eqref{1D EL equation}, that is, $\textup{Ker}(\mathcal{L}_\infty)=\textup{span}\{\partial_zQ_\infty\}$ (see \cite{K} and \cite[Proposition 2.8]{W}). Moreover, it is coercive in a restricted function space. Specifically, there exists $C_L>0$ such that   
\begin{equation}\label{1d linearized operator coercivity}
\langle\mathcal{L}_\infty \phi,  \phi\rangle_{L_z^2(\mathbb{R})}\geq C_L \|\phi\|_{L_z^2(\mathbb{R})}^2
\end{equation}
for all radially symmetric $\phi\in L_z^2(\mathbb{R})$, such that $\langle \phi,Q_\infty\rangle_{L_z^2(\mathbb{R})}=0$ (see \cite[Propositions 2.7 and 2.8]{W}). 

For a sufficiently large $\omega\geq 1$, let $Q_\omega$ be an energy minimizer of the variational problem $\mathcal{J}_\omega(m)$ obtained in Theorem \ref{th1} and consider the 3d linearized operator 
\begin{equation}\label{linearized operator}
\mathcal{L}_\omega=\omega(H_y-2)-\partial_z^2+\mu_\omega-3Q_\omega^2
\end{equation}
acting on $L_x^2(\mathbb{R}^3)$ with the domain $\{u\in H^2_x(\R^3): |y|^2 u\in L^2_x(\R^3)\}$. From the dimension reduction limit (Theorem \ref{h1cv}), it is expected that the operator $\mathcal{L}_\omega$ satisfies properties similar to those of the 1d operator $\mathcal{L}_\infty$.

First, we show that the 3d linearized operator has analogous coercivity.

\begin{proposition}[Coercivity of the linearized operator $\mathcal{L}_\omega$]\label{coe}
Let $C_L>0$ be the constant given in \eqref{1d linearized operator coercivity}. Then, for a sufficiently large $\omega\geq 1$, we have
\begin{align*}
&\left\langle \mathcal{L}_\omega\varphi  ,\varphi \right\rangle_{L_x^2(\R^3)}  \ge \frac{C_L}{4}\|P_0\varphi\|_{H^1_x(\R^3)}^2 +\frac{\omega(\Lambda_1-2)}{2}\|P_1 \varphi\|_{L^2_x(\R^3)}^2
\end{align*}
for all $\varphi\in \Sigma$ such that $\varphi(x)=\varphi(|y|,|z|)$ and $\langle Q_\omega, \varphi\rangle_{L^2_x(\mathbb{R}^3)}=0$.
\end{proposition}

\begin{proof}
Suppose that $\varphi$ satisfies the assumptions of the proposition. Then, because $\mu_\omega\rightarrow \mu_\infty$, $Q_\omega\rightarrow Q_\infty(z) \Phi_0(y)$ in $L^\infty_x(\R^3)$ and $\langle (H_y-2)\varphi,\varphi \rangle_{L_x^2(\R^3)}\ge (\Lambda_1-2)\|P_1\varphi\|_{L_x^2(\R^3)}^2$, it suffices to show that
$$
\langle \tilde{\mathcal{L}}_\infty \varphi  ,\varphi\rangle_{L_x^2(\R^3)}  \ge \frac{C_L}{3}\|P_0\varphi\|_{H^1_x(\R^3)}^2 -\frac{\omega(\Lambda_1-2)}{2}\|P_1 \varphi\|_{L^2_x(\R^3)}^2,
$$
where
$$\tilde{\mathcal{L}}_\infty= -\partial_z^2+\mu_\infty-3(Q_\infty(z) \Phi_0(y))^2$$
is an auxiliary 3d linear operator on $L_x^2(\mathbb{R}^3)$. 

By decomposing $\varphi= P_0\varphi+P_1\varphi$, where $P_0\varphi=\varphi_\parallel(z)\Phi_0(y)$ and $\varphi_\parallel(z)=\langle \varphi(\cdot,z),\Phi_0\rangle_{L_y^2(\mathbb{R}^2)}$, we write
\begin{align*}
\big\langle \tilde{\mathcal{L}}_\infty \varphi  ,\varphi\big\rangle_{L_x^2(\R^3)}&\ge \big\langle \tilde{\mathcal{L}}_\infty (\varphi_\parallel(z)\Phi_0(y)), \varphi_\parallel(z)\Phi_0(y)\big\rangle_{L_x^2(\R^3)}-6\big\langle  \left(Q_\infty  \Phi_0 \right)^2 P_0\varphi  ,P_1\varphi \big\rangle_{L_x^2(\R^3)} \\
&\quad   -3 \big\langle  \left(Q_\infty  \Phi_0 \right)^2  P_1\varphi  ,P_1\varphi \big\rangle_{L_x^2(\R^3)},
\end{align*}
where the cross term $\langle (-\partial_z^2+\mu_\infty)P_0\varphi,P_1\varphi\rangle_{L_x^2(\R^3)}$ is canceled, and the non-negative term $\langle (-\partial_z^2+\mu_\infty)P_1\varphi,P_1\varphi\rangle_{L_x^2(\R^3)}$ is dropped. Then, by the H\"older and the Cauchy-Schwarz inequalities, we can show that 
\begin{align*}
&\Big|6\big\langle  \left(Q_\infty  \Phi_0 \right)^2  P_0\varphi  ,P_1\varphi \big\rangle_{L_x^2(\R^3)}+3\big\langle  \left(Q_\infty  \Phi_0 \right)^2 P_1\varphi  ,P_1\varphi \big\rangle_{L_x^2(\R^3)}\Big|\\
&\le o_\omega(1)\|P_0\varphi\|_{L^2_x(\R^3)}^2+\frac{(\Lambda_1-2)\omega}{4}\|P_1\varphi\|_{L^2_x(\R^3)}^2
\end{align*}
because $Q_\infty$ is bounded. We also note that integrating out the $y$-variable 
$$\big\langle \tilde{\mathcal{L}}_\infty (\varphi_\parallel(z)\Phi_0(y)), \varphi_\parallel(z)\Phi_0(y)\big\rangle_{L_x^2(\R^3)}=\langle \mathcal{L}_\infty \varphi_\parallel,\varphi_\parallel\rangle_{L^2_z(\R)},$$
where $\mathcal{L}_\infty$ is the 1d linearized operator (see \eqref{1d linearized operator}). Therefore, the proof of the proposition can be further reduced to show the lower bound:
\begin{equation}\label{reduced coercivity estimate}
\langle \mathcal{L}_\infty \varphi_\parallel,\varphi_\parallel\rangle_{L^2_z(\R)}\ge  \frac{C_L}{2}  \|\varphi_\parallel\|_{H^1_z(\R)}^2 -\|P_1 \varphi\|_{L^2_x(\R^3)}^2.
\end{equation}

We show \eqref{reduced coercivity estimate} using the coercivity \eqref{1d linearized operator coercivity} of the 1d linearized operator $\mathcal{L}_\infty$. To do so, we denote the component of $\varphi_\parallel$ orthogonal to $Q_\infty$ by
$$\tilde{\varphi}_\parallel=\varphi_\parallel-\frac{\langle \varphi_\parallel, Q_\infty\rangle_{L^2_z(\R)}}{\|Q_\infty\|_{L^2_z(\R)}^2} Q_\infty.$$
We observe that under the orthogonality condition $\langle  \varphi,Q_\omega \rangle_{L^2_x(\R^3)}=0$ and the convergence $Q_\omega\rightarrow Q_\infty(z)\Phi_0(y)$ in $H^1_x(\R^3)$, $\varphi_\parallel$ is almost orthogonal to $Q_\infty$ as follows:
\begin{align*}
\langle \varphi_\parallel,Q_\infty\rangle_{L^2_z(\R)}&=\big\langle \varphi_\parallel(z)\Phi_0(y),Q_\infty(z)\Phi_0(y)\big\rangle_{L^2_x(\R^3)}=\big\langle \varphi,Q_\infty(z)\Phi_0(y)\big\rangle_{L^2_x(\R^3)}\\
&=\langle  \varphi,Q_\omega \rangle_{L^2_x(\R^3)}+\big\langle  \varphi,Q_\infty(z)\Phi_0(y)-Q_\omega \big\rangle_{L^2_x(\R^3)}\\
&=o_\omega(1)\|\varphi\|_{L^2_x(\R^3)}.
\end{align*}
Consequently, $\varphi_\parallel=\tilde{\varphi}_\parallel+o_\omega(1)\|\varphi\|_{L^2_x(\R^3)} Q_\infty$. Hence, by the H\"older and the Cauchy-Schwarz inequalities, the core part is extracted as
\begin{align*}
\langle \mathcal{L}_\infty  \varphi_\parallel, \varphi_\parallel\rangle_{L^2_z(\R)}&=\langle \mathcal{L}_\infty  \tilde{\varphi}_\parallel, \tilde{\varphi}_\parallel\rangle_{L^2_z(\R)}+o_\omega(1)\|\varphi\|_{L^2_x(\R^3)}  \langle\mathcal{L}_\infty  Q_\infty ,\tilde{\varphi}_\parallel  \rangle_{L^2_z(\R)}\\
&\quad +o_\omega(1)\|\varphi\|_{L^2_x(\R^3)}^2 \langle \mathcal{L}_\infty  Q_\infty, Q_\infty\rangle_{L^2_z(\R)}\\
&\ge  \langle \mathcal{L}_\infty  \tilde{\varphi}_\parallel, \tilde{\varphi}_\parallel\rangle_{L^2_z(\R)}+o_\omega(1)\|\varphi\|_{L^2_x(\R^3)}^2.
\end{align*}
Then, it follows from the coercivity of $\mathcal{L}_\infty$ (see \eqref{1d linearized operator coercivity}) that
$$\langle \mathcal{L}_\infty  \varphi_\parallel, \varphi_\parallel\rangle_{L^2_z(\R)}\ge C_L\|\tilde{\varphi}_\parallel\|_{H^1_z(\R)}^2+o_\omega(1)\| \varphi\|_{L^2_x(\R^3)}^2.$$
Thus, using $\|\tilde{\varphi}_\parallel\|_{H^1_z(\R)}^2\ge \|\varphi_\parallel\|_{H^1_z(\R)}^2+o_\omega(1)\|\varphi\|_{L^2_x(\R^3)}^2$, we prove \eqref{reduced coercivity estimate}.
\end{proof}

Next, we show the non-degeneracy of the energy minimizer $Q_\omega$. The non-degeneracy will not be used to prove the uniqueness of the minimizer; however, it is included for interest and potential future applications.

\begin{proposition}[Non-degeneracy of minimizers for large $\omega$]
For sufficiently large $\omega>0$, the operator $\mathcal{L}_\omega$ is non-degenerate; that is, its kernel is given by
$$
\ker \mathcal{L}_\omega=\textup{span}\{\partial_z Q_\omega\}.
$$
\end{proposition}

\begin{proof}
For contradiction, we assume that for some large $\omega>0$, there exists $\varphi_\omega\in \{u\in H^2_x(\R^3): |y|^2 u\in L^2_x(\R^3)\}$ such that $\|\varphi_\omega\|_{L^2_x(\R^3)}=1, \ \langle \partial_z Q_\omega$, $ \varphi_\omega\rangle_{L^2_x(\R^3)}=0$, and $\mathcal{L}_\omega \varphi_\omega=0$. Then, we have
$$\begin{aligned}
3\int_{\mathbb{R}^3}Q_\omega^2|\varphi_\omega|^2dx&=\langle\mathcal{L}_\omega \varphi_\omega, \varphi_\omega\rangle_{L_x^2(\mathbb{R}^3)}+3\langle Q_\omega^2\varphi_\omega, \varphi_\omega\rangle_{L^2_x(\R^3)}\\
&=\omega \|\varphi_\omega\|_{\dot{\Sigma}_y}^2+\|\partial_z \varphi_\omega\|_{L^2_x(\R^3)}^2+\mu_\omega \|\varphi_\omega\|_{L^2_x(\R^3)}^2\\
&\geq \mu_\omega=\mu_\infty+o_\omega(1),
\end{aligned}$$
while, by the Sobolev inequality with Lemma \ref{hig}, 
$$\begin{aligned}
\int_{\mathbb{R}^3}Q_\omega^2|\varphi_\omega|^2dx&=\sum_{k=-\infty}^\infty\int_{T_k}Q_\omega^2|\varphi_\omega|^2dx \le \sum_{k=-\infty}^\infty\|Q_\omega\|_{L^\infty_x(T_k)}^2\|\varphi_\omega\|_{L^2_x(T_k)}^2\\
&\lesssim\sum_{k=-\infty}^\infty\|Q_\omega\|_{H^2_x(T_k)}^2\|\varphi_\omega\|_{L^2_x(T_k)}^2\leq \sup_k\|\varphi_\omega\|_{L^2_x(T_k)}^2\cdot\sum_{k=-\infty}^\infty\|Q_\omega\|_{H^2_x(T_k)}^2\\
&=\sup_k\|\varphi_\omega\|_{L^2_x(T_k)}^2\cdot\|Q_\omega\|_{H^2_x(\mathbb{R}^3)}^2\lesssim \sup_k\|\varphi_\omega\|_{L^2_x(T_k)}^2,
\end{aligned}$$
where $T_k=\R^2\times [k,k+1), k\in \Z$ and the implicit constants are independent of $\omega$. Hence, by combining these two inequalities, we obtain 
\begin{equation}\label{key estimate for non-degeneracy}
\frac{\mu_\infty}{2}\leq \omega \|\varphi_\omega\|_{\dot{\Sigma}_y}^2+\|\partial_z \varphi_\omega\|_{L^2_x(\R^3)}^2+\mu_\omega \|\varphi_\omega\|_{L^2_x(\R^3)}^2\lesssim \sup_k\|\varphi_\omega\|_{L^2_x(T_k)}^2\leq 1.
\end{equation}
Thus, by translating $\varphi_\omega(y,z-k_\omega)$ using suitable $k_\omega\in\mathbb{Z}$ if necessary, but still denoting by $\varphi_\omega$, we may assume that 
\begin{equation}\label{ww1}
\liminf_{\omega\rightarrow \infty}\|\varphi_\omega\|_{L^2_x(\R^2\times [0,1))}^2\geq\frac{\mu_\infty}{2}.
\end{equation}
On the other hand, by \eqref{key estimate for non-degeneracy}, the sequence $\{\varphi_\omega\}_\omega$ is bounded in $\Sigma$ and thus passing to a subsequence, $\varphi_\omega \rightharpoonup \varphi_\infty\neq0$ in $\Sigma$ as $\omega\to\infty$. Moreover, by \eqref{key estimate for non-degeneracy}, $\|P_1\varphi_\omega\|_{L^2_x(\R^3)}^2\lesssim \|P_1\varphi_\omega\|_{\dot{\Sigma}_y}^2=\|\varphi_\omega\|_{\dot{\Sigma}_y}^2\lesssim \frac{1}{\omega}\to 0$, and consequently, by \eqref{ww1},
$$\varphi_{\omega,\parallel}(z)=\int_{\R^2}\varphi_\omega(y,z) \Phi_0(y)dy\rightharpoonup\varphi_{\infty,\parallel}(z)=\int_{\R^2}\varphi_\infty(y,z) \Phi_0(y)dy\neq 0$$
in $L_z^2(\mathbb{R})$. Therefore, by collecting and using the dimension reduction limit (Theorem \ref{h1cv}), we get that for any $g\in C_c^\infty(\mathbb{R})$ 
$$\begin{aligned}
0&=\langle\mathcal{L}_\omega\varphi_\omega, g(z)\Phi_0(y)\rangle_{L_x^2(\mathbb{R}^3)}\\
&=\big\langle (-\partial_z^2+\mu_\infty-3(Q_\infty(z) \Phi_0(y))^2)\varphi_{\omega,\parallel}(z)\Phi_0(y),g(z)\Phi_0(y) \big\rangle_{L_x^2(\mathbb{R}^3)}+o_\omega(1)\\
&=\langle\mathcal{L}_\infty \varphi_{\infty,\parallel}, g\rangle_{L_z^2(\mathbb{R})}+o_\omega(1),
\end{aligned}$$
in other words, $\varphi_{\infty,\parallel}\in \textup{Ker}(\mathcal{L}_\infty)$. However, since 
$$\begin{aligned}
0=\langle\varphi_\omega, \partial_z Q_\omega \rangle_{L_x^2(\mathbb{R}^3)}&=\langle\varphi_{\infty,\parallel}(z)\Phi_0(y), \partial_zQ_\infty(z)\Phi_0(y) \rangle_{L_x^2(\mathbb{R}^3)}+o_\omega(1)\\
&=\langle\varphi_{\infty,\parallel}, \partial_zQ_\infty\rangle_{L_z^2(\mathbb{R})}+o_\omega(1),
\end{aligned}$$
this contradicts the non-degeneracy of the minimizer $Q_\infty$.
\end{proof}

\subsection{Proof of Theorem \ref{uniqu}}
We now prove the uniqueness of the minimizer using the coercivity of its linearized operator (Proposition \ref{coe}). For contradiction, we assume that the variational problem $\mathcal{J}_\omega(m)$ has two different minimizers $Q_\omega(x)=Q_\omega(|y|,|z|)$ and $\tilde{Q}_\omega(x)=\tilde{Q}_\omega(|y|,|z|)$. We introduce the functional
$$
I_\omega(u)= E_\omega(u)+\frac{\mu_\omega}{2}M(u).
$$
Then, it is obvious that $I_\omega(Q_\omega)=I_\omega(\tilde{Q}_\omega)$.
We decompose
$$
\tilde{Q}_\omega=\sqrt{1-\delta_\omega^2}Q_\omega+R_\omega,
$$
where $\langle Q_\omega, R_\omega \rangle_{L_x^2(\mathbb{R}^3)}=0$. 
Then, since $m=\| \tilde{Q}_\omega\|_{L_x^2(\R^3)}^2=(1-\delta_\omega^2)m+\|R_\omega\|_{L_x^2(\R^3)}^2$ and $\|Q_\omega-\tilde{Q}_\omega\|_{L_x^2(\R^3)}^2=(1-\sqrt{1-\delta_\omega^2})^2m+\|R_\omega\|_{L_x^2(\R^3)}^2=o_\omega(1)$, we have
\begin{equation}\label{epsi0}
\delta_\omega=\frac{1}{\sqrt{m}}\|R_\omega\|_{L_x^2(\R^3)}=o_\omega(1).
\end{equation}
By inserting $\tilde{Q}_\omega=\sqrt{1-\delta_\omega^2}Q_\omega+R_\omega$ into $I_\omega(\tilde{Q}_\omega)$, we reorganize the terms in the increasing order of $R_\omega$ as
\begin{align*}
I_\omega(\tilde{Q}_\omega)&=\frac{\omega}{2}\big\|\sqrt{1-\delta_\omega^2}Q_\omega+R_\omega\big\|_{\dot{\Sigma}_y}^2+\frac{1}{2}\big\|\partial_z (\sqrt{1-\delta_\omega^2}Q_\omega+R_\omega)\big\|_{L_x^2(\mathbb{R}^3)}^2\\
&\quad -\frac{1}{4}\big\|\sqrt{1-\delta_\omega^2}Q_\omega+R_\omega\big\|_{L_x^4(\mathbb{R}^3)}^4+\frac{\mu_\omega}{2}\big\|\sqrt{1-\delta_\omega^2}Q_\omega+R_\omega\big\|_{L_x^2(\mathbb{R}^3)}^2 \\
&=\frac{1-\delta_\omega^2}{2}\Big\{\omega \left\|Q_\omega\right\|_{\dot{\Sigma}_y}^2+ \|\partial_z Q_\omega\|_{L_x^2(\mathbb{R}^3)}^2+\mu_\omega\|Q_\omega\|_{L_x^2(\mathbb{R}^3)}^2\Big\}-\frac{(1-\delta_\omega^2)^2}{4}\|Q_\omega\|_{L_x^4(\mathbb{R}^3)}^4\\
&\quad +\big\langle \sqrt{1-\delta_\omega^2}(\omega(H_y-2)-\partial_z^2+\mu_\omega)Q_\omega-(1-\delta_\omega^2)^\frac32Q_\omega^3 ,R_\omega \big\rangle_{L_x^2(\R^3)}\\
&\quad +\frac12\big\langle (\omega(H_y-2)-\partial_z^2+\mu_\omega)R_\omega-3(1-\delta_\omega^2)Q_\omega^2R_\omega ,R_\omega \big\rangle_{L_x^2(\R^3)}\\
&\quad -\sqrt{1-\delta_\omega^2}\langle Q_\omega , R_\omega^3\rangle_{L_x^2(\R^3)}- \frac14\|R_\omega\|_{L_x^4(\mathbb{R}^3)}^4.
\end{align*} 
For the zeroth and first-order terms, applying the equation \eqref{3D EL equation}, we write
\begin{align*}
&\frac{1-\delta_\omega^2}{2}\Big\{\omega \left\|Q_\omega\right\|_{\dot{\Sigma}_y}^2+ \|\partial_z Q_\omega\|_{L_x^2(\mathbb{R}^3)}^2+\mu_\omega\|Q_\omega\|_{L_x^2(\mathbb{R}^3)}^2\Big\}-\frac{(1-\delta_\omega^2)^2}{4}\|Q_\omega\|_{L_x^4(\mathbb{R}^3)}^4\\
&=\frac{1}{2}\Big\{\omega \left\|Q_\omega\right\|_{\dot{\Sigma}_y}^2+ \|\partial_z Q_\omega\|_{L_x^2(\mathbb{R}^3)}^2+\mu_\omega\|Q_\omega\|_{L_x^2(\mathbb{R}^3)}^2\Big\}-\left\{\frac{\delta_\omega^2}{2}+\frac{(1-\delta_\omega^2)^2}{4}\right\}\|Q_\omega\|_{L_x^4(\mathbb{R}^3)}^4\\
&=I_\omega(Q_\omega)- \frac{\delta_\omega^4}{4} \|Q_\omega\|_{L_x^4(\mathbb{R}^3)}^4=I_\omega(Q_\omega)+o_\omega(1)\|R_\omega\|_{L_x^2(\mathbb{R}^3)}^2
\end{align*} 
and
\begin{align*}
&\big\langle \sqrt{1-\delta_\omega^2}(\omega(H_y-2)-\partial_z^2+\mu_\omega)Q_\omega-(1-\delta_\omega^2)^\frac32Q_\omega^3 ,R_\omega \big\rangle_{L_x^2(\R^3)}\\
&=\sqrt{1-\delta_\omega^2}\big\langle (\omega(H_y-2)-\partial_z^2+\mu_\omega)Q_\omega-Q_\omega^3 ,R_\omega \big\rangle_{L_x^2(\R^3)}\\
&\quad+\big(\sqrt{1-\delta_\omega^2}-(1-\delta_\omega^2)^\frac32\big)\langle Q_\omega^3  ,R_\omega \rangle_{L_x^2(\R^3)}\\
&=\delta_\omega^2\sqrt{1-\delta_\omega^2}\langle Q_\omega^3  ,R_\omega \rangle_{L_x^2(\R^3)}=o_\omega(1)\|R_\omega\|_{L_x^2(\mathbb{R}^3)}^2,
\end{align*} 
where \eqref{epsi0} is used in the last step in both cases. For the second-order terms, by extracting the lineaarized operator, we write
$$\begin{aligned}
&\big\langle (\omega(H_y-2)-\partial_z^2+\mu_\omega)R_\omega-3(1-\delta_\omega^2)Q_\omega^2R_\omega ,R_\omega \big\rangle_{L_x^2(\R^3)}\\
&=\langle \mathcal{L}_\omega R_\omega, R_\omega\rangle_{L_x^2(\mathbb{R}^3)}+3\delta_\omega^2 \langle Q_\omega^2R_\omega ,R_\omega\rangle_{L_x^2(\R^3)}\\
&=\langle \mathcal{L}_\omega R_\omega, R_\omega\rangle_{L_x^2(\mathbb{R}^3)}+o_\omega(1)\|R_\omega\|_{L_x^2(\mathbb{R}^3)}^2.
\end{aligned}$$
For the higher-order terms, we observe from Theorem \ref{h1cv} and the Sobolev embedding $L_x^\infty(\mathbb{R}^3)\hookrightarrow H_x^2(\mathbb{R}^3)$ that $\|R_\omega\|_{L_x^\infty(\mathbb{R}^3)}\to 0$. Hence, we have
$$\sqrt{1-\delta_\omega^2}|\langle Q_\omega, R_\omega^3\rangle_{L_x^2(\R^3)}|+\frac{1}{4}\|R_\omega\|_{L_x^4(\mathbb{R}^3)}^4=o_\omega(1)\|R_\omega\|_{L_x^2(\mathbb{R}^3)}^2.$$
Putting it all together, we obtain 
$$\begin{aligned}
I_\omega(\tilde{Q}_\omega)&=I_\omega(Q_\omega) +\frac12\langle \mathcal{L}_\omega R_\omega, R_\omega\rangle_{L_x^2(\mathbb{R}^3)}+o_\omega(1)\|R_\omega\|_{L_x^2(\mathbb{R}^3)}^2.
\end{aligned}$$
Then, Proposition \ref{coe} yields a contradiction.

\appendix

\section{Dimension reduction for the Cauchy problem}\label{dimcau}

We establish the convergence of general 3d solutions in Proposition \ref{prop: global existence} as the confinement is strengthened, which corresponds to the downward arrow on the right-hand side of Figure 1. While this dimension reduction is not used to prove the main result, it is included because it might be of its own interest.

\begin{theorem}[Dimension reduction from the 3d to the 1d NLS]\label{thm: dimension reduction for NLS}
For $\omega \ge (C_{GN})^4m^2$, where $C_{GN}$ is given in Lemma \ref{GN type inequality}, we assume that $u_{\omega,0}\in \Sigma$, $M(u_{\omega,0})=m$,
$$E_\omega(u_{\omega,0})<0\textup{ and }\|u_{\omega,0}\|_{\dot{\Sigma}_y}^2\leq \sqrt{\omega}.$$
Let $u_\omega(t)$ be the global solution to the 3d NLS \eqref{NLS} with initial data $u_{\omega,0}$, and let $v_\omega(t)$ be the global solution to the 1d NLS \eqref{1dNLS} with initial data $u_{\omega,\parallel}(0)=\langle u_{\omega,0},\Phi_0\rangle_{L_y^2(\mathbb{R}^2)}$. Then, there exists constants $C_1, C_2>0$, independent of $\omega$, such that 
$$\|u_\omega(t,x)-v_\omega(t,z)\Phi_0(y)\|_{L_x^2(\mathbb{R}^3)}\leq\frac{C_1}{\sqrt{\omega}}e^{C_2t}.$$
\end{theorem}

For the proof, we require suitable bounds that are uniform-in-$\omega$ for nonlinear solutions. We recall that the 1d NLS \eqref{1dNLS} is mass-subcritical and is globally well-posed. Moreover, by summing the space-time norm bounds on short-time intervals with the mass conservation law, one can derive the following bound (see \cite{Caz}): 
\begin{lemma}
A global solution $v(t)\in C_t(\mathbb{R}; L_z^2(\mathbb{R}))$ to the 1d NLS \eqref{1dNLS} with initial data $v_0\in L_z^2(\mathbb{R})$ satisfies 
\begin{equation}\label{1d cubic NLS bound}
\|v(t)\|_{L_t^4([-T,T]); L_z^\infty(\mathbb{R})}\lesssim T^{\frac{1}{4}}\|v_0\|_{L_z^2(\mathbb{R})}\quad\textup{for all }T\geq 1.
\end{equation}
\end{lemma}

Because $\|u_{\omega,\parallel}(0)\|_{L_z^2(\mathbb{R})}^2\leq \|u_{\omega,0}\|_{L_x^2(\mathbb{R}^3)}^2=m$, the 1d solution $v_\omega(t)$ satisfies a bound of the form \eqref{1d cubic NLS bound}. Based on what follows, we claim that the core $\Phi_0(y)$-directional component of the 3d solution satisfies the same bound. We decompose 
$$u_\omega(t,x)=u_{\omega,\parallel}(t,z)\Phi_0(y)+(P_1u_\omega)(t,x),$$
where 
\begin{equation}\label{v tilde(t)}
u_{\omega,\parallel}(t,z):=\langle u_\omega(t,y,z),\Phi_0(y)\rangle_{L_y^2(\mathbb{R}^2)}.
\end{equation}
Here, $u_{\omega,\parallel}(t,z)$ is essential because by Proposition \ref{prop: global existence} $P_1u_\omega(t)$ becomes negligible as $\omega\to\infty$.

\begin{lemma}\label{uniform bound}
Under the assumptions in Theorem \ref{thm: dimension reduction for NLS}, let $u_{\omega,\parallel}(t)$ be given by \eqref{v tilde(t)}. Then, we have
\begin{equation}\label{v tilde bound}
\|u_{\omega,\parallel}(t)\|_{L_t^4([-T,T]; L_z^\infty(\mathbb{R}))}\lesssim T^{\frac{1}{4}}
\end{equation}
for all $T\geq1$, where the implicit constant depends on the mass and energy of $u_{\omega,0}$ but independent of $\omega$.
\end{lemma}

\begin{proof}
By direct calculations, we observe that 
$u_{\omega,\parallel}(t,z)$ solves
\begin{equation}\label{tilde v(t) equation}
\begin{aligned}
i\partial_tu_{\omega,\parallel}&=\left\langle (\omega(H_y-2)-\partial_z^2)u_\omega-|u_\omega|^2u_\omega,\Phi_0\right\rangle_{L_y^2(\mathbb{R}^2)}\\
&=-\partial_z^2u_{\omega,\parallel}-\left\langle|P_0u_\omega|^2P_0u_\omega,\Phi_0\right\rangle_{L_y^2(\mathbb{R}^2)}-\left\langle|u_\omega|^2u_\omega-|P_0u_\omega|^2P_0u_\omega,\Phi_0\right\rangle_{L_y^2(\mathbb{R}^2)}\\
&=-\partial_z^2u_{\omega,\parallel}-\frac{1}{2\pi}|u_{\omega,\parallel}|^2u_{\omega,\parallel}-\left\langle|u_\omega|^2u_\omega-|P_0u_\omega|^2P_0u_\omega,\Phi_0\right\rangle_{L_y^2(\mathbb{R}^2)},
\end{aligned}
\end{equation}
where the last step, we used $P_0u_\omega(t,x)=u_{\omega,\parallel}(t,z)\Phi_0(y)$ and $\|\Phi_0\|_{L_y^4(\mathbb{R}^2)}^4=\frac{1}{2\pi}$. Equivalently, we have
\begin{align*}
u_{\omega,\parallel}(t)&=e^{it\partial_z^2}u_{\omega,\parallel}(0), \\
&\quad +i\int_{0}^t e^{i(t-s)\partial_z^2}\bigg\{\frac{1}{2\pi}|u_{\omega,\parallel}|^2u_{\omega,\parallel}+\left\langle|u_\omega|^2u_\omega-|P_0u_\omega|^2P_0u_\omega,\Phi_0\right\rangle_{L_y^2(\mathbb{R}^2)}\bigg\}(s)ds.
\end{align*}
For a sufficiently small $T\in(0,1]$, we let $I=[-T,T]$. Then, the well-known 1d Strichartz estimate \cite{Caz}
\begin{equation}\label{1d Strichartz}
\|e^{it\partial_z^2}\varphi\|_{L_t^4(\mathbb{R};L_x^\infty(\mathbb{R}))}\lesssim\|\varphi\|_{L_z^2(\mathbb{R})}
\end{equation}
yields
$$\begin{aligned}
\|u_{\omega,\parallel}\|_{L_t^4(I;L_x^\infty(\mathbb{R}))}&\lesssim\|u_{\omega,\parallel}(0)\|_{L_z^2(\mathbb{R})}+\||u_{\omega,\parallel}|^2u_{\omega,\parallel}\|_{L_t^1(I;L_z^2(\mathbb{R}))}\\
&\quad +\left\|\left\langle|u_\omega|^2u_\omega-|P_0u_\omega|^2P_0u_\omega,\Phi_0\right\rangle_{L_y^2(\mathbb{R}^2)}\right\|_{L_t^1(I;L_z^2(\mathbb{R}))}.
\end{aligned}$$
Hence, by the H\"older inequality with 
$$\big||u_\omega|^2u_\omega-|P_0u_\omega|^2P_0u_\omega\big|\lesssim \Big\{|P_0u_\omega|^2+|P_1u_\omega|^2\Big\} |P_1u_\omega|\lesssim |u_\omega|^2|P_1u_\omega|$$
for the last term, it follows that 
$$\begin{aligned}
\|u_{\omega,\parallel}\|_{L_t^4(I;L_x^\infty(\mathbb{R}))}&\lesssim \|u_{\omega,\parallel}(0)\|_{L_z^2(\mathbb{R})}+T^{\frac{1}{2}}\|u_{\omega,\parallel}\|_{L_t^4(I;L_z^\infty(\mathbb{R}))}^2 \|u_{\omega,\parallel}\|_{C_t(I;L_z^2(\mathbb{R}))}\\
&\quad +T\|u_\omega\|_{L_t^\infty(I;L_x^4(\mathbb{R}^3))}^2 \|P_1u_\omega\|_{C_t(I;L_x^2(\mathbb{R}^3))}.
\end{aligned}$$
We observe that by the Gagliardo--Nirenberg inequality (Lemma \ref{GN type inequality}) and the bound \eqref{eq: global existence}, $\|u_\omega(t)\|_{L_x^4(\mathbb{R}^3)}$ is uniformly bounded in $\omega$ and $\|P_1u_\omega(t)\|_{L_x^2(\mathbb{R}^3)}\leq\frac{1}{\sqrt{\omega}}$. Thus, we obtain 
\begin{equation}\label{key estimate for dimension reduction}
\|u_{\omega,\parallel}\|_{L_t^4(I;L_x^\infty(\mathbb{R}))}\lesssim 1+T^{\frac{1}{2}}\|u_{\omega,\parallel}\|_{L_t^4(I;L_z^\infty(\mathbb{R}))}^2+T\sqrt{\tfrac{(C_{GN})^4m^2}{\omega}}, 
\end{equation}
where the implicit constant depends on the mass and energy of $u_{\omega,0}$. Thus, taking a sufficiently small $T>0$, we prove that $\|u_{\omega,\parallel}\|_{L_t^4(I;L_x^\infty(\mathbb{R}))}\lesssim 1$. Note that the time interval is selected depending only on the mass and energy. Thus, by iterating, we prove the desired bound.
\end{proof}

\begin{proof}
By Proposition \ref{prop: global existence}, it suffices to estimate $r_\omega(t)=(u_{\omega,\parallel}-v_\omega)(t)$. Subtracting 1d NLS \eqref{1dNLS} from equation \eqref{tilde v(t) equation}, we see that the difference $r_\omega(t)$ solves
$$i\partial_tr_\omega=-\partial_z^2r_\omega-\frac{1}{2\pi}\left(|u_{\omega,\parallel}|^2u_{\omega,\parallel}-|v_\omega|^2v_\omega\right)-\left\langle|u_\omega|^2u_\omega-|P_0u_\omega|^2P_0u_\omega,\Phi_0\right\rangle_{L_y^2(\mathbb{R}^2)}$$
with no initial data. Thus, we have
$$\begin{aligned}
\|r_\omega(t)\|_{L_x^2(\mathbb{R})}&\leq\int_0^t\big\||u_{\omega,\parallel}|^2u_{\omega,\parallel}(s)-|v_\omega|^2v_\omega(s)\big\|_{L_z^2(\mathbb{R})}ds\\
&\qquad+\int_0^t\left\|\left\langle|u_\omega|^2u_\omega(s)-|P_0u_\omega|^2P_0u_\omega(s),\Phi_0\right\rangle_{L_y^2(\mathbb{R}^2)}\right\|_{L_z^2(\mathbb{R})}ds.
\end{aligned}$$
Using the estimates from the proof of \eqref{key estimate for dimension reduction} for the second integral, we obtain 
$$\|r_\omega(t)\|_{L_x^2(\mathbb{R})}\lesssim \sqrt{\frac{(C_{GN})^4m^2}{\omega}}|t|+\int_0^t \Big\{\|v_\omega(s)\|_{L_z^\infty(\mathbb{R})}^2+\|u_{\omega,\parallel}(s)\|_{L_z^\infty(\mathbb{R})}^2\Big\}\|r_\omega(s)\|_{L_z^2(\mathbb{R})}ds.$$
Therefore, Gronwall's inequality together with \eqref{1d cubic NLS bound} and \eqref{v tilde bound} yields the desired convergence estimate.
\end{proof}

\end{document}